\documentclass[12pt,reqno]{amsart}
\usepackage{amssymb, a4wide, amsmath, mathtools,xy}
\xyoption{all}

\newtheorem{theorem}{Theorem}[section]
\newtheorem{definition}[theorem]{Definition}
\newtheorem{corollary}[theorem]{Corollary}
\newtheorem{example}[theorem]{Example}

\newtheorem{remark}[theorem]{Remark}
\newtheorem{proposition}[theorem]{Proposition}
\newtheorem{lemma}[theorem]{Lemma}

\def\al{\alpha}

\def\vp{\varphi}

\def\pa{\partial}

\def\gm{\gamma}

\def\lra{\longrightarrow}

\def\ol{\overline}

\def\Ker{\operatorname{Ker}}
\def\Im{\operatorname{Im}}
\def\Coker{\operatorname{Coker}}

\DeclareMathOperator{\Der}{Der}
\DeclareMathOperator{\Hom}{Hom}
\DeclareMathOperator{\End}{End}
\def\D{\operatorname{D}}

\def\id{\operatorname{id}}
\def\cl{\operatorname{cl}}
\DeclareMathOperator{\A}{\sf {A}}
\DeclareMathOperator{\M}{\sf {M}}
\DeclareMathOperator{\E}{\sf {E}}
\DeclareMathOperator{\V}{\sf {V}}
\DeclareMathOperator{\AWB}{\sf {AWB}}
\DeclareMathOperator{\XAWB}{\sf {XAWB}}
\DeclareMathOperator{\IAWB}{\sf {IAWB}}
\DeclareMathOperator{\AAA}{\mathcal{A}}
\DeclareMathOperator{\MM}{\mathcal{M}}
\DeclareMathOperator{\EE}{\mathcal{E}}
\DeclareMathOperator{\ab}{\mathrm{ab}}

\DeclareMathOperator{\Ext}{Ext}
\DeclareMathOperator{\XExt}{XExt}

\DeclareMathOperator{\B}{\sf B}
\DeclareMathOperator{\C}{\sf C}

\DeclareMathOperator{\Z}{\mathcal{ Z}}
\DeclareMathOperator{\e}{\sf e}

\newcommand{\K}{\mathbb{K}}

\numberwithin{equation}{section}
\title[Wells sequence and crossed extensions of algebras with bracket]{Wells type exact sequence and crossed extensions of algebras with bracket}

\author[J. M. Casas]{Jos\'e Manuel Casas$^1$}
\address{$^1$Department of Applied Mathematics I, and CITMAga, E. E. Forestal,  University of Vigo, 36005 Pontevedra, Spain}
\email{jmcasas@uvigo.es}

\author[E. Khmaladze]{Emzar Khmaladze$^2$}
\address{$^2$The University of Georgia, Kostava St. 77a, 0171 Tbilisi, Georgia \& A. Razmadze Mathematical Institute of Tbilisi State University,
	Tamarashvili St. 6, 0177 Tbilisi, Georgia   }
\email{e.khmaladze@ug.edu.ge}

\author[M. Ladra]{Manuel Ladra$^3$}
\address{$^3$Department of Mathematics, CITMAga, Universidade de Santiago de Compostela, 15782 Santiago de Compostela, Spain}
\email{manuel.ladra@usc.es}



\begin{document}

\begin{abstract}
	We study the extensibility problem of a pair of derivations associated with an abelian extension of algebras with bracket, and derive an exact sequence of the Wells type.	
	We introduce crossed modules for algebras with bracket and prove their equivalence with internal categories in the category of algebras with bracket.
	We interpret the set of equivalence classes of crossed extensions as the second cohomology.
	Finally, we construct an eight-term exact sequence in the cohomology of algebras with bracket.
\end{abstract}

\maketitle

\section{Introduction}

The concept of a new algebraic structure called algebra with bracket was introduced in \cite{CP} as a generalization of a (non-commutative) Poisson algebra.
It is an associative algebra equipped with a bracket operation $[-,-]$ related to the multiplication by the following fundamental relation:
\[
[ab,c]=[a,c]b+a[b,c].
\]
The origins of such generalization of Poisson algebras can be found in the physics literature (see, e.g. \cite{Ka}).

Based on the careful analyses of the free objects in the category of algebras with bracket and using the Hochschild complex for associative algebras, an explicit cochain complex computing Quillen cohomology of algebras with bracket is constructed in \cite{CP}. In the same paper, zero and first cohomologies are characterized by means of derivations and abelian extensions of algebras with bracket, respectively. Later, in \cite{Ca}, a homology with trivial coefficients of algebras with bracket is developed with applications to universal central extensions, and among other results, the five-term exact homology and cohomology sequences are obtained.

In this paper, we aim to investigate the following three tasks in the context of algebras with bracket:

\noindent {\bf 1}. \emph{To study the extensibility problem for derivations associated with an abelian extension.}

\noindent In the case of groups, such an extensibility problem of automorphisms goes back to Baer~\cite{Ba}.
Important progress was archived by Wells in extensions of abstract groups \cite{We}, where a map, later called a Wells map, is defined and an exact sequence, also called a Wells exact sequence,  connecting various automorphism groups is constructed.
The Wells map and the extensibility of a pair of automorphisms associated with an (abelian) extension were studied in the context
of various algebraic structures \cite{BaSi,DuTa,HaHa,HZ,JiLi,TaXu}.
In parallel,  the extensibility of derivations associated with an abelian extension of Lie algebras and associative algebras
have been studied in \cite{BaSa} and \cite{TX}, respectively.

\noindent {\bf 2}. \emph{To introduce crossed modules and to describe the second cohomology via crossed extensions.}

\noindent Crossed modules of groups introduced in \cite{Wh} are algebraic models of path-connected CW-spaces, whose homotopy groups are trivial in all dimensions greater than $2$.  Crossed modules also play an essential role in algebra because any crossed module determines an element in the third cohomology of groups. This was firstly observed for group cohomology by Mac Lane-Whitehead \cite{MaWh}.
The similar fact that a cohomology class can be considered as an equivalence class of crossed modules is known in Lie algebras, associative algebras, Leibniz algebras, Leibniz $n$-algebras and so on \cite{BaMi,CKL1, Cu,Wa}.

\noindent {\bf 3}. \emph{To construct an eight-term exact cohomology sequence}.

\noindent
The five-term exact sequence in low-dimensional cohomology of algebras with bracket obtained in \cite{Ca} and  our description of  the second cohomology via crossed
extensions leads us to extend this five-term sequence in three new terms, involving the second cohomology vector spaces, similarly to the case of cohomology sequences in groups, Lie algebras, Leibniz algebras and so on \cite{Ca2,CeAz,Ra}.

\subsection*{Organization} After the introductory Section 1, the paper is organized into five more sections.  Section 2 recalls basic definitions on algebras with bracket, construction and some results of their cohomology. In section 2, the extensibility problem for derivations of algebras with bracket associated with an abelian extension is stated and investigated, the Wells map is constructed
(Definition~\ref{Def_Well_map}), which is used to determine the necessary and sufficient condition for the extensibility of a pair of derivations associated to a given abelian extension  (Theorem~\ref{theorem extensible});  the Wells exact sequence connecting various vector spaces of derivations is also obtained (Theorem~\ref{theorem Well seq}, Corollary~\ref{Cor Well seq}). In Section 4, we introduce the notion of crossed module of algebras with bracket (Definition~\ref{def crossed module}) and show the equivalence with internal categories (Theorem~\ref{Theo inter cat}). In Section 5, the crossed modules are used to define crossed extensions  of algebras with bracket. We establish a bijection between the second cohomology and the set of equivalence classes of crossed extensions (Theorem~\ref{T11}). The last Section 6 is devoted to the construction of an eight-term exact sequence in the cohomology of algebras with bracket (Theorem~\ref{Theo 8 term }).

\subsection*{Conventions and notation} Throughout the paper we  fix a ground  field $\mathbb{K}$. All vector spaces and algebras are $\mathbb{K}$-vector spaces and $\mathbb{K}$-algebras, and linear maps are $\mathbb{K}$-linear maps as well. $\Hom$ and $\otimes$ denote $\Hom_\mathbb{K}$ and $\otimes_\mathbb{K}$, respectively.
For the composition of two maps $f$ and $g$, we write either $g\circ f$ or simply $gf$. For the identity map on
a set $X$ we use the notation $\id_{X}$.
For any equivalence relation on a set $X$ we shall write  $\cl(x)$ to denote the equivalence class of an element $x\in X$.

\

\section{Algebras With Bracket} \label{AWB}

\subsection{Basic definitions} \label{Basic}
\begin{definition}[\cite{CP}] An  algebra with bracket, or an \textup{AWB} for short,  is  an associative (not necessarily commutative) algebra ${\A}$
	equipped with a bilinear map (bracket operation) $[-,-] \colon  {\A} \times {\A} \to {\A}$, $(a, b)\mapsto [a,b]$ satisfying the following identity:
	\begin{equation}\label{FE}
		[a  b,c] = [a,c]  b + a  [b,c]
	\end{equation}
	for all $a, b, c \in {\sf A}$.
\end{definition}

\emph{A homomorphism} of \textup{AWB}'s is a homomorphism of associative algebras preserving the bracket operation.
We denote by ${\sf AWB}$ the respective category of \textup{AWB}'s.

\begin{example}\label{ejemplos}\
	\begin{enumerate}
		\item[(i)]  Any vector space $\A$ with the trivial multiplication and bracket, i.e. $ab=0$ and $[a,b]=0$ for all $a,b\in \A$,  is an \textup{AWB}, called \emph{an abelian \textup{AWB}}.
		
		\item[(ii)]  Any Poisson algebra is an \textup{AWB}. In fact, the category ${\sf Poiss}$ of commutative Poisson algebras is a subcategory of ${\sf AWB}$. The inclusion functor $\sf Poiss \hookrightarrow AWB$ has as left adjoint the functor given by
		${\A} \mapsto {\A}_{\sf Poiss}$, where ${\A}_{\sf Poiss}$ is the maximal quotient of ${\A}$, such that the following relations hold: $a  b-b   a \sim 0$, $[a,a]\sim 0$ and $[a,[b,c]]+[b,[c,a]]+[c,[a,b]] \sim 0$.
		\item[(iii)]  Let ${\A}$ be an associative algebra equipped with a linear
		map $D\colon  {\A} \to {\A}$. Then ${\A}$ is an \textup{AWB} concerning the bracket given by
		\[
		[a,b] \coloneqq aD(b) - D(b)a.
		\]
		For $D=0$, one obtains an \textup{AWB}, which has the trivial bracket. The more interesting is the case
		when $D = \id$, then ${\A}$ is an algebra with bracket concerning the usual bracket
		\[
		[a,b]\coloneqq ab - ba.
		\]
		This particular algebra with bracket is called the tautological algebra with bracket associated to an associative algebra ${\A}$.
		Let us recall that this last bracket also defines a Lie algebra structure on the underlying vector space of $\A$.
		
		\item[(iv)] If $({\A}, \prec, \succ)$ is a dendriform algebra (see \cite{Lo}), then ${\A}$ with the product and bracket defined respectively by
		$a b = a \prec b + a \succ b$ and  $[a,b] = a \prec b + a \succ b - b \prec a - b \succ a$ is an \textup{AWB}.
		
		\item[(v)] Let ${\A}$ be an \textup{AWB}, then the vector space ${\A} \otimes {\A}$ can be endowed with the structure of \textup{AWB} in the following two ways:
		\begin{itemize}
			\item[1.] 	\begin{align*}
				(a_1 \otimes a_2)  (b_1 \otimes b_2) & = (a_1b_1) \otimes (a_2b_2), \\
				[a_1 \otimes a_2, b_1 \otimes b_2] & =[a_1,[b_1,b_2]] \otimes a_2 + a_1 \otimes [a_2,[b_1,b_2]].
			\end{align*}
			\item[2.]
			\begin{align*}
				(a_1 \otimes a_2)(b_1 \otimes b_2) & = a_1 \otimes (a_2 (b_1 b_2)), \\
				[a_1 \otimes a_2, b_1 \otimes b_2] & = a_1 \otimes [a_2,b_1  b_2] + [a_1,b_1  b_2] \otimes a_2.
			\end{align*}
		\end{itemize}
		for all $a_1, a_2, b_1, b_2 \in {\A}$.
	\end{enumerate}
\end{example}

Let us recall the following notions from \cite{Ca}.
\emph{A subalgebra} ${\B}$ of an \textup{AWB} ${\A}$ is a vector subspace  which is closed under the product and the bracket operation, that is, ${\B}~{\B} \subseteq
{\B}$ and $[{\B},{\B}] \subseteq {\B}$.
A subalgebra ${\B}$ is said to be a  \emph{right (respectively, left) ideal} if  ${\A}~{\B} \subseteq {\B}$, $[{\A},{\B}] \subseteq {\B}$ (respectively,
${\B}~{\A} \subseteq {\B}$, $[{\B},{\A}] \subseteq {\B}$).
If  ${\B}$ is both left and right ideal, then it is said to be  a \emph{two-sided ideal}. In this case, the quotient
${\A}/{\B}$ is endowed with an \textup{AWB} structure  naturally induced from the operations on ${\A}$.

Let $\A$ be an \textup{AWB} and ${\B}, {\C}$ be two-sided ideals of ${\A}$. The \emph{commutator ideal} of ${\B}$ and ${\C}$ is the two-sided ideal of
${\B}$ and ${\C}$
\[
[[{\B},{\C}]] = \langle \{bc, cb, [b,c], [c,b] \mid b \in {\B}, c \in {\C}
\} \rangle.
\]
Obviously $[[{\B}, {\C}]] \subseteq {\B} \bigcap {\C}$. Observe that  $[[{\B},{\C}]]$ is not a two-sided ideal of
${\A}$ in general, except when ${\B} = {\A}$ or ${\C} = {\A}$. In the particular case $ {\B} = {\C} = {\A}$, one obtains the definition of derived algebra of ${\A}$,
\[
[[{\A},{\A}]]=\langle \{aa', [a,a'] \mid a, a' \in {\A} \} \rangle.
\]
Note that the quotient ${\A}/[[{\A},{\A}]] $ is an abelian \textup{AWB} and will be denoted by $\A_{\ab}$.

The \emph{center} of an \textup{AWB} ${\A}$ is the two-sided ideal
\[
{\Z}({\A}) =  \{a \in {\A} \mid ab
= 0 = ba,[a,b] = 0 = [b,a], {\rm for \ all}\ b \in {\sf A}\}.
\]
Note that an algebra with bracket ${\A}$ is abelian if and only if ${\A} = {\Z}({\A})$.

\subsection{Actions and semi-direct product}

\begin{definition} \label{action def}
	Let ${\A}$ and ${\M}$ be two \textup{AWB}'s. An action of ${\A}$ on ${\M}$ consists of four bilinear maps
	\[
	\begin{array}{llcll}
		{\A} \times {\M} \to {\M}, & (a, m) \mapsto {a \cdot m}, & & {\M} \times {\A}  \to {\M}, &(m, a) \mapsto m \cdot a,\\
		{\A} \times {\M} \to {\M}, & (a, m) \mapsto \{{a, m}\}, & & {\M} \times {\A}  \to {\M}, & (m, a) \mapsto \{m,a\},\\
	\end{array}
	\]
	such that the following conditions hold:
	\begin{equation}\label{equations_action}
		\begin{array}{lcl}
			(a_1 a_2) \cdot m = a_1 \cdot ({a_2} \cdot m), && \{{a_1 \cdot} m , { a_2}\} = {{a_1 \cdot}} \{ m,   a_2 \} + [a_1, a_2] \cdot m,\\
			m \cdot {(a_1 a_2)}  =(m \cdot {a_1}){ \cdot {a_2}} , && \{m{ \cdot {a_1}}, { {a_2}}\} =  \{m, {{a_2}}\}{ \cdot {a_1}} + m{ \cdot [a_1, a_2]},\\
			({{a_1 \cdot}} m ){ \cdot {a_2}}  = {{a_1 \cdot}} (m{ \cdot {a_2}} ) , && \{a_1 a_2, m\} = {a_1 \cdot} \{a_2, m \} + \{a_1,  m\} \cdot {a_2},\\
			(m_1 m_2){ \cdot a} = m_1 (m_2{ \cdot {\ a}}), && [m_1{ \cdot { a}},m_2] = m_1\{ a, m_2\} + [m_1, m_2]{ \cdot a},\\
			{{a \cdot}} (m_1 m_2)  = ({{a \cdot}}  m_1 ) m_2 , && [{{a \cdot}} {m_1}, m_2] =  a \cdot [m_1, m_2] + \{a, m_2\} m_1,\\
			(m_1{ \cdot  a}) m_2  = m_1 ({{a \cdot}} {m_2} ) , && \{m_1 m_2, a\} = m_1\{m_2, a\} + \{m_1, a\} m_2,\\
		\end{array}
	\end{equation}
	for all $a, a_1, a_2 \in {\A}$, $m, m_1, m_2 \in {\M}$. The action is called trivial if all these bilinear maps are trivial.
\end{definition}

Let us remark that if an action of an \textup{AWB} $\A$ on an abelian \textup{AWB} $\M$ is given,
then all six equations in the last three lines of \eqref{equations_action} vanish and we get the definition of a \emph{representation} $\M$ of $\A$ (see \cite{CP}).

\begin{example}\label{action}\
	\begin{enumerate}
		\item[(i)] If $\M$ is a representation of an \textup{AWB} $\A$ thought as an abelian \textup{AWB}, then there is an action of $\A$ on the abelian \textup{AWB} $\M$.
		
		\item[(ii)] If ${\M}$ is a two-sided ideal of an \textup{AWB} ${\A}$, then the operations in ${\A}$ yield an action of ${\A}$ on ${\M}$, that is ${a \cdot}m= am$, $m{\cdot a} = m a$, $\{a,m\}= [a,m]$, $\{m,a\}= [m,a]$, for all $m \in {\M}$, and $a \in {\A}$.
		
		\item[(iii)] If $ 0\to {\M} \stackrel{i}\to {\B} \stackrel{\pi}\to {\A }\to 0$ is a split short exact sequence of algebras with bracket, that is, there exists a homomorphism of algebras with bracket  $s\colon  {\A}\to {\B}$ such that $\pi \circ s = \id_{{\A}}$, then there is an action of  ${\A}$ on ${\M}$, given by:
		\[
		\begin{array}{lcl}
			a \cdot m = i^{-1}\big( s(a) i(m)\big), && m \cdot a = i^{-1}\big( i(m) s(a) \big), \\
			\{a , m\} = i^{-1}\big(\left[s(a), i(m) \right]\big), && \{m, a\} = i^{-1}\big(\left[i(m), s(a)\right]\big),
		\end{array}
		\]
		\noindent for any $a\in {\A}$, $m\in {\M}$.
		
		\item[(iv)] Any homomorphism of \textup{AWB}'s $f \colon  {\A} \to {\M} $  induces an action of ${\A}$ on ${\M}$ in the
		standard way by taking images of elements of ${\A}$ and operations in  ${\M}$.
		
		\item[(v)] If $\mu \colon {\M}\to {\A}$ is a surjective homomorphism of \textup{AWB}'s and the kernel of $\mu$ is contained in the center of $\M$, i.e. $\Ker(\mu) \subseteq {\Z}({\M})$, then there is an action of ${\A}$ on ${\M}$, defined in the standard way, i.e. by choosing pre-images of elements of ${\A}$ and taking operations in ${\M}$.
	\end{enumerate}
\end{example}

\begin{definition}
	Let ${\A}$ and ${\M}$ be \textup{AWB}'s with an action of ${\A}$ on ${\M}$. The semi-direct product of ${\M}$ and ${\A}$, denoted by ${\M} \rtimes {\A}$, is the \textup{AWB} whose  underlying vector space is ${\M} \oplus {\A}$ endowed with the operations
	\begin{align*}
		(m_1, a_1) (m_2, a_2) & = \big(m_1 m_2 + {{a_1 \cdot}} m_2 + m_1{ \cdot a_2}, a_1 a_2 \big), \\
		[(m_1, a_1), (m_2, a_2)] & = \big([m_1, m_2] + \{a_1, m_2\} + \{m_1, a_2\}, [a_1, a_2]\big)
	\end{align*}
	for all $m_1, m_2 \in {\M}$, $a_1, a_2 \in {\A}$.
\end{definition}

Given an action of an \textup{AWB} $\A$ on $\M$, straightforward calculations show that the sequence of \textup{AWB}'s
\begin{equation}
	0 \lra {\M} \stackrel{i}\lra {\M} \rtimes {\A} \stackrel{\pi}\lra {\A} \lra 0
\end{equation}
where $i(m)=(m, 0), \pi(m,a)=a$, is exact. Moreover ${\M}$ is a two-sided ideal of ${\M} \rtimes {\A}$ and this sequence splits by $s \colon  {\A} \to {\M} \rtimes {\A}, s(a) = (0, a)$. Then, as in Example~\ref{action} {(iii)}, the above sequence induces another action of ${\A}$ on $\M$ given by
\begin{align*}
	a \cdot m &= i^{-1}\big((0,a) (m,0)\big), \qquad m \cdot a = i^{-1}\big((m,0) (0,a)\big), \\
	\{a, m \} & = i^{-1}\big[(0,a), (m,0)\big],  \quad \{m, a\} = i^{-1}\big[(m,0), (0,a)\big],
\end{align*}
which actually matches the given one.

\subsection{Cohomology of \textup{AWB}'s} The Quillen cohomology of an \textup{AWB} $\A$ with coefficients in a
representation ${\M}$ of $\A$ is computed as the
cohomology of an explicit cochain complex $K^*(\A,\M)$ in \cite{CP}. Let us recall the main constructions.

Let $\A$ be an \textup{AWB} and $\M$  a representation of $\A$. Since $\M$ is a bimodule over the associative algebra $\A$, we can consider the Hochschild cochain complex $C^*(\A,\M)$  of $\A$ with coefficients in $\M$. Let us recall that $C^n(\A,\M) = \Hom(\A^{\otimes n}, \M) $ and the coboundary map $b^*$ is given by
\begin{align*}
	b^n (f)(a_0,\dots,a_{n}) = &  a_0 \cdot f(a_1,\dots,a_{n})  +
	\underset{\scriptsize{0\leq i\leq n-1}}\sum
	(-1)^{i+1} f(a_0,\dots,a_i a_{i+1},\dots,a_{n})\\
	& + (-1)^{n+1}f(a_0,\dots,a_{n-1})\cdot a_n\;.
\end{align*}

Now we let $\overline{C}^*(\A,\M)$ be the cochain complex defined by
\[
\overline{C}^n(\A,\M)={C}^{n+1}(\A,\M), \quad n\geq 0 \quad \text{and} \quad
\overline{C}^n(\A,\M)=0, \quad n \le 0.
\]

At the same time, the vector space  $\M^e = \Hom(\A,\M)$ has a bimodule structure over $\A$ given by
\[
(a_0\cdot f)(a_1)=a_0\cdot f(a_1), \qquad (f\cdot a_0)(a_1)=f(a_1)\cdot a_0,
\]
for all $a_0,a_1\in \A$ and $f\in \M^e$.

There is a cochain map
\[
\alpha^n\colon  \overline{C}^n(\A,\M) \to \overline{C}^n(\A,\M^e), \quad n\geq 0
\]
defined by
\[
\left( \alpha^0(f)(a_0)\right)(a_1) =\{a_0,f(a_1)\}-f[a_0,a_1]+\{f(a_0),a_1\}
\]
provided $n=0$ and for $n\geq 1$ by
\[
\big(\alpha^n(f)(a_0,\dots, a_n)\big)(a_{n+1})
=\{f(a_0,\dots, a_n), a_{n+1}\}-
\underset{\scriptsize{0\leq i\leq n}}\sum
f\left(a_0, \dots, [a_i,a_{i+1}], \dots , a_{n+1}\right) .
\]

Then the complex $K^*(\A,\M)$ is defined
to be the cone of the cochain map $\alpha$. By definition, the $n$-th \textup{AWB} cohomology $H^{*}_{ \sf AWB}(\A,\M)$ of $\A$ with coefficients in $\M$ is the $(n-1)$-th cohomology of the complex $K^*(\A,\M)$, i.e.   $ H^{n}_{\sf AWB}(\A,\M) = H^{n-1}K^*(\A,\M)$, $n\geq 0$.

\

For future references, we specify below the cochains and coboundary
maps of $K^*(\A,\M)$ in low dimensions:

\noindent $\bullet$ $K^{-1}(\A,\M)$ consists of all linear maps $h\colon \A \to \M$

\noindent $\bullet$  $K^0(\A,\M)$ consists of all pairs $(f,g)$, where $f\colon \A^{\otimes 2}\to \M$
(resp. $g\colon \A\to \M^e$) is a $2$-cochain (resp. $1$-cochain ) in the Hochschild complex $C^*(\A,\M)$ (resp. $C^*(\A,\M^e)$).

\noindent $\bullet$  $K^1(\A,\M)$ consists of all pairs $(\mathsf{f},\mathsf{g})$, where $\mathsf{f}\colon \A^{\otimes 3}\to \M$ 
(resp. $\mathsf{g} \colon \A ^{\otimes 2} \to \M^e$) is a $3$-cochain (resp. $2$-cochain) in the Hochschild complex $C^*(\A,\M)$ (resp. $C^*(\A,\M^e)$);

\noindent $\bullet$  The  coboundary map $\partial^{-1}$ is given by
\begin{align}\label{eq_pa-1}
	&\partial^{-1}(h)(a_0,a_1)= \left( -b^1(h)(a_0,a_1), \ - \alpha^0(h)(a_0)(a_1) \right) \\
	& \quad = \big( -a_0\cdot h(a_1) +h(a_0a_1)-h(a_0)\cdot a_1, \ -\{a_0,h(a_1)\}+h[a_0,a_1]-\{h(a_0),a_1\}  \big). \notag
\end{align}

\noindent $\bullet$  The  coboundary map $\partial^0$ is given by
\[
\partial^0(f,g)= \left( -b^2(f), b^1(g) - \alpha^1(f) \right),
\]
where
\begin{align}\label{eq_0_cocycle_1}
	b^2(f)(a_0,a_1,a_2) = a_0\cdot f(a_1,a_2)-f(a_0a_1, a_2)+f(a_0,a_1a_2)-f(a_0,a_1)\cdot a_2,
\end{align}
and
\begin{align}\label{eq_0_cocycle_2}
	\left(  b^1(g) - \alpha^1(f) \right) (a_0,a_1,a_2) & =  a_0\cdot  \big( g(a_1)(a_2)\big)  + \big( g(a_0)(a_2)\big)\cdot  a_1 -  g(a_0a_1)  (a_2) \\ \notag
	& \quad 	- \Big(\{f(a_0,a_1), a_2\}- f\big( [a_0,a_2],a_1\big) - f\big( a_0, [a_1,a_2]\big) \Big).
\end{align}

\noindent $\bullet$  The coboundary map $\partial^1$ is given by
\[
\partial^1(\mathsf{f},\mathsf{g})= \left( -b^3(\mathsf{f}), b^2(\mathsf{g}) - \alpha^2(\mathsf{f}) \right)  .
\]
Therefore, $1$-cocycles in $K^*(\A,\M)$  are those  pairs $(\mathsf{f},\mathsf{g})\in K^1(\A,\M) $ such that $\mathsf{f}\colon \A^{\otimes 3}\to \M$ is a $2$-cocycle in the Hochschild complex $C^*(\A,\M)$, that is,
\begin{align}\label{eq_1_cocycle_1}
	a_0\cdot \mathsf{f}(a_1,a_2,a_3)-\mathsf{f}(a_0a_1, a_2, a_3)+\mathsf{f}(a_0,a_1a_2, a_3)-\mathsf{f}(a_0,a_1, a_2a_3) + \mathsf{f}(a_0,a_1, a_2)\cdot a_3 =0,
\end{align}
and $\mathsf{g}\colon \A^{\otimes 2} \to \M^e$ is a linear map satisfying the following equality
\begin{align}\label{eq_1_cocycle_2}
	& \{\mathsf{f}(a_0,a_1,a_2), a_3\}- \mathsf{f}\left( [a_0,a_3],a_1,a_2\right)  - \mathsf{f}\left( a_0, [a_1,a_3],a_2\right)  - \mathsf{f}\left( a_0,a_1,[a_2,a_3]\right) \\ \notag
	& \	=  a_0 \cdot \left( \mathsf{g}(a_1,a_2)(a_3)\right)  - \mathsf{g}(a_0a_1,a_2)(a_3)+ \mathsf{g}(a_0, a_1a_2)(a_3) - \left( \mathsf{g}(a_0, a_1)(a_3) \right) \cdot a_2 \, .
\end{align}

\subsection{$H^0_{\AWB}$ and $H^1_{\AWB}$}
Given an \textup{AWB} $\A$ and its representation $\M$,
$H^0_{\AWB}(\A,\M)$ and $H^1_{\AWB}(\A,\M)$ are described in \cite{CP} via \textup{AWB}-derivations and abelian extensions, respectively.
Let us briefly recall the relevant concepts related to these results, which will be useful later.

Let $\A$ be an \textup{AWB} and $\M$  a representation of $\A$.
An \textup{AWB}-\emph{derivation} is a linear map $d\colon \A\to\M$ such that
$d(a_0a_1)=a_0d(a_1)+d(a_0)a_1$ (i.e. $d$ is a derivation in the sense of associative algebras)
and $d[a_0,a_1]=[a_0,d(a_1)]+[d(a_0),a_1]$ for all $a_0, a_1 \in \A$.
Let $\Der_{\AWB}(\A,\M)$ denote the vector space of all \textup{AWB}-derivations. By \cite[Lemma 4.4]{CP} there is an isomorphism of vector spaces
\begin{equation}\label{iso h0}
	H^0_{\sf AWB}(\A,\M)\cong \Der_{\AWB}(\A,\M).
\end{equation}

In the particular case when $\M=\A$, we denote $\Der_{\AWB}(\A,\A)$ by $\Der_{\AWB}(\A)$.
It is easy to check that $d_1d_2 - d_2d_1 \in \Der_{\AWB}(\A)$ if $d_1, d_2 \in \Der_{\AWB}(\A)$,
showing that $\Der_{\AWB}(\A)$ is a Lie subalgebra of the Lie algebra of all derivations  $\Der(\A)$ of the associative algebra $\A$.

\

Any abelian extension of an \textup{AWB} $\A$ by a representation $\M$ of $\A$, that is, a short exact sequence of \textup{AWB}'s
\begin{equation*}
	E\colon   0\lra \M \overset{i}{\lra} \E\overset{p}{\lra} \A\lra 0,
\end{equation*}
which induces the given representation structure on $\M$, gives rise to a $0$-cocycle $(f,g)\in K^0(\A,\M) $ by choosing a linear section
$s\colon \A\to \M$ of $p$, and by defining $f\colon \A^{\otimes 2}\to \M$ and $g\colon \A\to\Hom(\A,\M)$ as follows
\begin{align}\label{eq_1cocycle(f,g)}
	f(a_0, a_1) & =i^{-1} \big( s(a_0)s(a_1) - s(a_0a_1) \big) , \\
	g(a_0)(a_1) & =i^{-1} \big( [s(a_0),s(a_1)] - s[a_0,a_1] \big) \notag
\end{align}
for all $a_0, a_1\in \A$. This gives a well-defined bijection between the set of equivalence classes $\Ext_{\AWB}(\A,\M)$
of such abelian extensions of $\A$ by $\M$ and the first cohomology of $\A$ with coefficients in $\M$ \cite[Lemma 4.6]{CP}, i.e.
\begin{equation}\label{iso ext h1}
	H^1_{\AWB}(\A,\M)\cong \Ext_{\AWB}(\A,\M).
\end{equation}

Let us note that this bijection allows us to endow the set  ${\Ext}(\A,\M)$ with a vector space structure induced from the one of $H_{\AWB}^1(\A,\M)$.
The respective addition in ${\Ext}_{\AWB}(\A,\M)$ is defined by the well-known ``Baer sum'' of extensions and the scalar multiplication
is defined for any $\lambda\in \K$ by $\lambda \cl(E)= \cl{(\lambda E)}$,
where $ \lambda E\colon  0\to \M \xrightarrow{\lambda \cdot i} \E  \xrightarrow{\; p \; }  \A \to 0$.

\section{\textup{AWB}-derivations of abelian extensions}

\subsection{Extensibility of \textup{AWB}-derivations}

In this section, we consider the extensibility problem of \textup{AWB}-derivations associated with an abelian extension.
\begin{definition}
	Let $0\lra \M \overset{i}{\lra} \E\overset{p}{\lra} \A\lra 0$ be an abelian extension of an \textup{AWB} $\A$ by its representation $\M$.
	A pair $(d_{\M},d_{\A})\in \Der_{\AWB}(\M) \times \Der_{\AWB}(\A)$ is called extensible if there is a derivation $d_{\E}\in \Der_{\AWB}(\E) $
	such that $d_{\E} \circ i= i \circ d_{\M}$ and $p \circ d_{\E}= d_{A} \circ p$.
\end{definition}
Let us remark that, here $\M$ is considered as an abelian \textup{AWB} and hence,
a derivation $d_{\M}\in \Der_{\AWB}(\M) $ is just a linear map, and $\Der_{\AWB}(\M) = \End_{\mathbb{K}}(\M)$.

\begin{lemma}\label{lemma_semi-dir}
	Let $\M$ be a representation of an \textup{AWB} $\A$. Consider $\M$ as an abelian \textup{AWB}.
	Let $d_{\A}\in \Der_{\AWB}(\A)$  and $d_{\M}\in \Der_{\AWB}(\M)$. Then $(d_{\M},d_{\A})\in \Der_{\AWB}(\M\rtimes \A)$
	if and only if the following equalities hold for all $a\in \A$ and $m\in \M$
	\begin{align}
		d_{\M}(a\cdot m) & =a\cdot d_{\M}(m)+d_{\A}(a)\cdot m,\label{eq_d1}\\
		d_{\M}(m\cdot a) & =m\cdot d_{\A}(a)+d_{\M}(m)\cdot a,\label{eq_d2} \\
		d_{\M}\{a,m\} & =\{a,d_{\M}(m)\}+\{d_{\A}(a), m\},\label{eq_d3} \\
		d_{\M}\{m, a\} & =\{m, d_{\A}(a)\}+\{d_{\M}(m), a\}.\label{eq_d4}
	\end{align}
\end{lemma}
\begin{proof}
	The proof requires only direct calculations using the definition of \textup{AWB}-derivations.
\end{proof}

For any \textup{AWB} $\A$ and a representation $\M$ of $\A$, let us denote by
\[
\D_{\AWB}(\M, \A) =  \big\{(d_{\M},d_{\A})\in \Der_{\AWB}(\M) \times \Der_{\AWB}(\A) \mid d_{\M} \ \text{and} \ d_{\A} \
\text{satisfy equations} \ \eqref{eq_d1}-\eqref{eq_d4}\big\}.
\]
One readily checks that
\[
(d_{\M}d'_{\M}-d'_{\M}d_{\M}, d_{\A}d'_{\A}-d'_{\A}d_{\A})\in \D_{{\AWB}}({\M}, {\A}) \ \text{if} \ (d_{{\M}}, d_{{\A}}), (d'_{\M}, d'_{\A}) \in \D_{\AWB}(\M, \A),
\]
which implies that  $\D_{\AWB}(\M, \A)$ is a Lie subalgebra of $\Der_{\AWB}(\M\rtimes \A)$.

Now we define a map
\begin{equation}\label{def_theta}
	\theta \colon  {\D}_{{\AWB}}({\M}, {\A})\times K^0({\A}, {\M}) \lra K^0({\A}, {\M})
\end{equation}
as follows. For any $(d_{\M},d_{\A})\in \D_{\AWB}(\M, \A)$ and any $0$-cochain $(f,g)\in K^0(\A, \M)$ we set
\[
\theta \big( (d_{\M},d_{\A}),(f,g)\big)=(f_{\theta}, g_{\theta}),
\]
where $f_{\theta}\colon \A^{\otimes 2} \to \M$ and $g_\theta\colon \A \to \Hom(\A, \M)$ are given by
\begin{align}\label{eq_def (f, g)theta}
	f_{\theta}(a_0,a_1) & =d_{\M}(f(a_0,a_1))- f(d_{\A}(a_0),a_1) - f(a_0,d_{\A}(a_1)) ,\\
	g_{\theta}(a_0)(a_1) & =d_{\M}(g(a_0)(a_1))- g(d_{\A}(a_0))(a_1) - g(a_0)(d_{\A}(a_1)).\notag
\end{align}

\begin{lemma}\label{Lemma theta cocycle}
	Let us fix $(d_{\M},d_{\A})\in \D_{\AWB}(\M, \A)$. Then we have:
	\begin{itemize}
		\item[(i)] If $(f,g)\in K^0(\A, \M)$ is a $0$-cocycle, then $(f_{\theta}, g_{\theta})$ is a $0$-cocycle as well.
		\item[(ii)] If $(f,g)$ and $(f',g')$ are two cohomologous cocycles, then  $(f_{\theta}, g_{\theta})$ and $(f'_{\theta}, g'_{\theta})$ are cohomologous as well.
	\end{itemize}
\end{lemma}
\begin{proof}
	(i) We need to check that $(f_{\theta}, g_{\theta})$ satisfies
	\begin{align}
		& a_0\cdot f_{\theta}(a_1,a_2)-f_{\theta}(a_0a_1, a_2)+f_{\theta}(a_0,a_1a_2)-f_{\theta}(a_0,a_1)\cdot a_2=0, \label{eq_f_theta}\\
		& a_0 \cdot \big( g_{\theta}(a_1)(a_2)\big)  + \big( g_{\theta}(a_0)(a_2)\big) \cdot a_1 -  g_{\theta}(a_0a_1)  (a_2) \label{eq_g_theta} \\
		& \qquad = \{f_{\theta}(a_0,a_1), a_2\}- f_{\theta}\big( [a_0,a_2],a_1\big) - f_{\theta}\big( a_0, [a_1,a_2]\big), \notag
	\end{align}
	for all $a_0, a_1, a_2\in \A$, whenever the same conditions hold for $(f,g)$.
	In effect, to check \eqref{eq_f_theta} we just use equations \eqref{eq_d1}, \eqref{eq_d2} and get
	\begin{align*}
		a_0\cdot f_{\theta}(a_1,a_2) & = d_{\M}(a_0\cdot f(a_1,a_2))-d_{\A}(a_0)\cdot f(a_1,a_2)-a_0\cdot f(d_{\A}(a_1),a_2) - a_0\cdot f(a_1, d_{\A}(a_2)),\\
		-f_{\theta}(a_0a_1, a_2) & = - d_{\M}(f(a_0a_1, a_2)) + f(a_0d_{A}(a_1), a_2) + f(d_{A}(a_0)a_1, a_2) + f(a_0a_1, d_{A}(a_2)),\\
		f_{\theta}(a_0,a_1 a_2) & =  d_{\M}(f(a_0,a_1a_2)) - f(d_{\A}(a_0), a_1a_2) - f(a_0, a_1d_{\A}(a_2)) - \ \ f(a_0, d_{\A}(a_1)a_2),\\
		-f_{\theta}(a_0,a_1)\cdot  a_2 & = -d_{M}(f(a_0,a_1)\cdot a_2) + f(a_0,a_1)\cdot d_{\A}(a_2)+ f(d_{\A}(a_0),a_1)\cdot a_2 + f(a_0,d_{\A}(a_2))\cdot a_2,
	\end{align*}
	and obviously the sum of all these expressions is equal to $0$. In the same way, by using equations \eqref{eq_d1}--\eqref{eq_d4} one checks immediately that \eqref{eq_g_theta} also holds.
	
	(ii) If $(f,g)$ and $(f', g')$ are cohomologous cocycles, then  there is a linear map $h:\A\to \M$ such that $(f-f', g-g') = \pa^{-1}(h)$, by \eqref{eq_pa-1} we get
	\begin{align*}
		(f-f')(a_0, a_1)  & = -a_0\cdot h(a_1) +h(a_0a_1)-h(a_0)\cdot a_1, \\
		(g-g')(a_0)(a_1)  & = -\{a_0,h(a_1)\}+h[a_0,a_1]-\{h(a_0),a_1\} .
	\end{align*}
	Then using only the definition of $(f_{\theta}, g_{\theta})$ in \eqref{eq_def (f, g)theta}, straightforward computations show that the linear map
	$d_{\M}h-hd_{\A}\colon \A \to \M$ satisfies the condition
	\begin{align*}
		(f_{\theta}-f'_{\theta}, g_{\theta}-g'_{\theta}) = \pa^{-1}(d_{\M}h-hd_{\A}).
	\end{align*}
	This completes the proof.
\end{proof}

\begin{remark}
	As a consequence of Lemma~\ref{Lemma theta cocycle}, note that the map $\theta$ in \eqref{def_theta} induces a bilinear map (denoted again by $\theta$)
	\begin{equation*}
		\theta \colon  {\D}_{{\AWB}}({\M}, {\A})\times H^1_{\AWB}(\A,\M) \lra H^1_{\AWB}(\A,\M).
	\end{equation*}
	Moreover, it is routine to check that $\theta$ defines a module structure on $H^1_{\AWB}(\A,\M)$ over the Lie algebra ${\D}_{\AWB}(\M, \A)$.
	
\end{remark}

Using Lemma~\ref{Lemma theta cocycle}, we state the following definition.
\begin{definition}\label{Def_Well_map}
	Let $E \colon  0\lra \M \overset{i}{\lra} \E\overset{p}{\lra} \A\lra 0$ be an abelian extension of an \textup{AWB} $\A$ by its representation $\M$ and $(f,g)\in K^0(\A,\M)$  the corresponding $0$-cocycle as in \eqref{eq_1cocycle(f,g)}. The map
	\[
	\omega \colon  {\D}_{\AWB}(\M, \A) \lra H^1_{\AWB}(\A,\M), \quad  \omega (d_{\M}, d_{\A})=\theta \big( (d_{\M}, d_{\A}), \cl(f,g) \big)=
	\cl( f_{\theta}, g_{\theta})
	\]
	is called the Wells map associated to the given abelian extension ${E}$.
\end{definition}

\begin{remark}\label{remark w=0}
	Note that if the abelian extension $E$ is split, that is, there is a homomorphism of \textup{AWB}'s $s\colon \A \to \E $ such that $p\circ s =\id _{\A}$, then $(f_{\theta},g_{\theta})=(f,g)=(0,0)$ and so $\omega$ is the trivial map.
\end{remark}

\begin{theorem}\label{theorem extensible}
	Let	$E \colon  0\lra \M \overset{i}{\lra} \E\overset{p}{\lra} \A\lra 0$ be an abelian extension of an \textup{AWB} $\A$ by its representation $\M$. A pair
	$(d_{\M},d_{\A})\in \Der_{\AWB}(\M) \times \Der_{\AWB}(\A)$ is extensible if and only if $(d_{\M},d_{\A})\in \D_{\AWB}(\M, \A)$ and $w(d_{\M},d_{\A})=0$.
\end{theorem}
\begin{proof}
	Let $s\colon \A\to \E$ be a linear section of $p$. Then any element of $\E$ has the form $i(m)+s(a)$,  where $m\in \M$ and $a\in \A$. We write $m$ instead of $i(m)\in \E$ to simplify notations.
	
	If $(d_{\M},d_{\A})\in \D_{\AWB}(\M, \A)$ and $w(d_{\M},d_{\A})=\theta \big( (d_{\A}, d_{\M}), \cl(f,g) \big)=
	\cl( f_{\theta}, g_{\theta}) = 0 $, there exists a linear map $h\colon \A\to \M$ such that  $( f_{\theta}, g_{\theta})= \pa^{-1}(h)$.
	By using \eqref{eq_pa-1} and \eqref{eq_def (f, g)theta} we get immediately:
	\begin{align}
		d_{\M}(f(a_0,a_1))- f(d_{\A}(a_0),a_1) - f(a_0,d_{\A}(a_1)) & = -a_0\cdot h(a_1) +h(a_0a_1)-h(a_0)\cdot a_1, \label{eq_f and h}\\
		d_{\M}(g(a_0)(a_1))- g(d_{\A}(a_0))(a_1) - g(a_0)(d_{\A}(a_1))  & =-\{a_0,h(a_1)\}+h[a_0,a_1]-\{h(a_0),a_1\}.\label{eq_g and h}
	\end{align}
	We define a map $d_{\E}\colon \E \to \E$ by
	\begin{equation}\label{def_d_E}
		d_{\E}(m+s(a))=d_{\M}(m)-h(a) +s(d_{\A}(a)).
	\end{equation}
	Obviously $d_{\E}$ is a linear map and
	\begin{align*}
		d_{\E} (i(m)) & = i (d_{\M}(m)),\\
		p(d_{\E}(m+s(a))) & =p (d_{\M}(m))-p(h(a)) +p(s (d_{\A}(a))) = d_{\A}(a) = d_{\A}(p(m+s(a))).
	\end{align*}
	To show that $d_{\E}$ is an \textup{AWB}-derivation, it suffices to check the following identities
	\begin{align}
		d_{\E}(s(a)m) & =d_{\E}(s(a))m+s(a)d_{\E}(m), \label{eq_d_E1}\\
		d_{\E}(ms(a)) & =d_{\E}(m)s(a)+md_{\E}(s(a)), \label{eq_d_E2}\\
		d_{\E}(s(a_0)s(a_1)) & =d_{\E}(s(a_0))s(a_1)+s(a_0)d_{\E}(s(a_1)), \label{eq_d_E3}\\
		d_{\E}[s(a),m] & =[d_{\E}(s(a)), m]+[s(a),d_{\E}(m)], \label{eq_d_E4} \\
		d_{\E}[m,s(a)] & =[d_{\E}(m),s(a)]+[m,d_{\E}(s(a))], \label{eq_d_E5}\\
		d_{\E}[s(a_0), s(a_1)] & =[d_{\E}(s(a_0)),s(a_1)]+[s(a_0),d_{\E}(s(a_1))]. \label{eq_d_E6}
	\end{align}
	The identities \eqref{eq_d_E1}, \eqref{eq_d_E2}, \eqref{eq_d_E4} and \eqref{eq_d_E5} easily follow from \eqref{eq_d1}, \eqref{eq_d2}, \eqref{eq_d3} and \eqref{eq_d4}, respectively.
	Justifications of \eqref{eq_d_E3} and \eqref{eq_d_E6}  are very similar and use \eqref{eq_f and h} and \eqref{eq_g and h}. We present only the calculations confirming \eqref{eq_d_E6}:
	\begin{align*}
		d_{\E}[s(a_0),s(a_1)] & \overset{\eqref{eq_1cocycle(f,g)}}{=} d_{\E} \big(g(a_0)(a_1)+ s[a_0,a_1]\big) \overset{\eqref{def_d_E}}{=} d_{\M} \big(g(a_0)(a_1)\big) - h[a_0,a_1]+ sd_{\A}[a_0,a_1]\\
		& \ \, = d_{\M} \big(g(a_0)(a_1)\big) - h[a_0,a_1]+ s[d_{\A}(a_0),a_1]+s[a_0,d_{\A}(a_1)]\\
		&\overset{\eqref{eq_1cocycle(f,g)}}{=} d_{\M} \big(g(a_0)(a_1)\big) - h[a_0,a_1] + [sd_{\A}(a_0),s(a_1)] - g(d_{\A}(a_0))(a_1) \\
		&   \qquad \qquad \qquad \qquad \ \ +[s(a_0),sd_{\A}(a_1)] - g(a_0)(d_{\A}(a_1))\\
		&\overset{\eqref{def_d_E}}{=} d_{\M} \big(g(a_0)(a_1)\big) - g(d_{\A}(a_0))(a_1)  - g(a_0)(d_{\A}(a_1)) - h[a_0,a_1] \\
		& \qquad  +  [d_{\E}(s(a_0)),s(a_1)] + [h(a_0), s(a_1)]+[s(a_0),d_{\E}(s(a_1))] + [s(a_0), h(a_1)]\\
		&\overset{\eqref{eq_g and h}}{=} -[a_0,h(a_1)]+h[a_0,a_1]-[h(a_0),a_1] - h[a_0,a_1]  +  [d_{\E}(s(a_0)),s(a_1)] \\
		& \qquad \qquad \qquad \qquad \ \ + [h(a_0),a_1]  +[s(a_0),d_{\E}(s(a_1))]+[a_0, h(a_1)]\\
		& \ =[d_{\E}(s(a_0)),s(a_1)]+[s(a_0),d_{\E}(s(a_1))].
	\end{align*}
	
	Conversely, suppose that $(d_{\M},d_{\A})\in \Der_{\AWB}(\M) \times \Der_{\AWB}(\A)$ is extensible,  i.e.  there exists $d_{\E}\in \Der_{\AWB}(\E) $ such that $d_{\E} \circ i= i \circ d_{\M}$ and $p \circ d_{\E}= d_{\A} \circ p$, then $d_{\M}$ is the restriction of $d_{\E}$ on $\M$, $d_{\M}= {d_{\E}}_{\mkern 1mu \vrule height 2ex\mkern2mu {\M}}$. Fix again a linear section $s\colon \A\to \E$ of $p$. Since $s(d_{\A}(a))- d_{\E}(s(a))\in \Ker p$ for all $a\in \A$, we have a linear map $h\colon  \A\to \M$ defined by
	\begin{equation}\label{coboundary h}
		h (a)=s(d_{\A}(a))- d_{\E}(s(a)).
	\end{equation}
	Now we claim that the pair $(d_{\M},d_{\A})$ satisfies the conditions \eqref{eq_d1}--\eqref{eq_d4}. For instance, we only present the proof for \eqref{eq_d3}:
	\begin{align*}
		d_{\M}\{a,m\}&=d_{\E}[s(a),m] = [d_{\E}(s(a)), m]+[s(a), d_{\E}(m)] \\
		& \! \! \! \overset{\eqref{coboundary h}}{=} [s (d_{\A}(a))- h (a), m] +[a, d_{\M}(m)]\\
		& =  [s (d_{\A}(a)), m] +\{a, d_{\M}(m) \}	=  \{ d_{\A}(a), m\} +\{a, d_{\M}(m)\}.
	\end{align*}
	The justifications of \eqref{eq_d1}, \eqref{eq_d2} and \eqref{eq_d4} are very similar. Hence, $(d_{\M},d_{\A})\in \D_{\AWB}(\M, \A)$.
	
	Finally, we show that
	$( f_{\theta}, g_{\theta})= \pa^{-1}(h)$, where $h$ is defined in \eqref{coboundary h} and  $w(d_{\M},d_{\A})=\theta \big( (d_{\A}, d_{\M}), \cl(f,g) \big)= \cl( f_{\theta}, g_{\theta})$. In effect, for any $a_0, a_1\in \A$ we have
	\begin{align}\label{final 1}
		d_{\E}\big(s(a_0)s(a_1) \big) &\overset{\eqref{eq_1cocycle(f,g)}}{=} d_{\E}\big(s(a_0a_1)+f(a_0,a_1) \big) \overset{\eqref{coboundary h}}{=} s(d_{\A}(a_0a_1))-h(a_0a_1)+d_{\M}\big(f(a_0,a_1)\big)  \\
		& \ 	= s\big(d_{\A}(a_0)a_1\big)+s\big(a_0d_{\A}(a_1)\big)-h(a_0a_1)+d_{\M}\big(f(a_0,a_1)\big) .\notag
	\end{align}
	On the other hand, since $d_{\E}$ is an \textup{AWB}-derivation, we get
	\begin{align}\label{final 2}
		d_{\E}\big(s(a_0)s(a_1) \big) &= d_{\E}\big(s(a_0)\big)s(a_1)+s(a_0)d_{\E}\big(s(a_1\big) \big)  \\
		& \!\! \! \overset{\eqref{coboundary h}}{=} 	\big( s(d_{\A}(a_0)) - h(a_0)\big) s(a_1) + s(a_0)\big(s(d_{\A}(a_1)) -h(a_1)\big)	\notag \\
		& = s\big(d_{\A}(a_0)\big) s(a_1) - h(a_0)s(a_1) + s(a_0)s\big(d_{\A}(a_1) \big) -s(a_0) h(a_1) \notag \\
		& = s\big(d_{\A}(a_0)\big) s(a_1) - h(a_0)\cdot a_1 + s(a_0)s\big(d_{\A}(a_1) \big) -a_0\cdot h(a_1) \notag.
	\end{align}
	Comparing the last lines in \eqref{final 1} and \eqref{final 2}, and taking in mind that
	\begin{align*}
		s\big(d_{\A}(a_0)\big) s(a_1) - s\big(d_{\A}(a_0)a_1\big) & = f(d_{\A}(a_0), a_1),\\
		s(a_0)s\big(d_{\A}(a_1) \big) - s\big(a_0d_{\A}(a_1)\big)  & = f(a_0, d_{\A}(a_1)),
	\end{align*}
	it is readily seen that
	\begin{align*}
		f_{\theta}(a_0,a_1)&=d_{\M}(f(a_0,a_1))- f(d_{\A}(a_0),a_1) - f(a_0,d_{\A}(a_1)) \\
		&= -a_0\cdot h(a_1) +h(a_0a_1)-h(a_0)\cdot a_1,
	\end{align*}
	One can repeat similar  computations to show that
	\begin{align*}
		g_{\theta}(a_0,a_1)&=   d_{\M}(g(a_0)(a_1))- g(d_{\A}(a_0))(a_1) - g(a_0)(d_{\A}(a_1)) \\
		&=-\{a_0,h(a_1)\}+h[a_0,a_1]-\{h(a_0),a_1\}.
	\end{align*}
	Hence $(f_{\theta}, g_{\theta} )= \pa^{-1}(h)$ and the proof is completed.
\end{proof}

\subsection{The Wells sequence for \textup{AWB}-derivations}
We continue using the previous subsection's notation. Given an abelian extension $E \colon  0\lra \M \overset{i}{\lra} \E\overset{p}{\lra} \A\lra 0$ of an \textup{AWB} $\A$ by its representation $\M$, we fix a linear section $s$ of $p$ and denote
\[
\Der_{\AWB}(\E|\M) = \{ d_{\E}\in \Der_{\AWB}(\E) \mid d_{\E}(\M)\subseteq \M \}.
\]
It is obviously a Lie subalgebra of $\Der_{\AWB}(\E)$. One readily checks that any element $d_{\E}\in \Der_{\AWB}(\E|\A)$ defines two \textup{AWB}-derivations
\[
d_{\E\mid\M}= {d_{\E}}_{\mkern 1mu \vrule height 2ex\mkern2mu {\M}}\in \Der_{\AWB}(\M)  \quad \text{and} \quad d^{\E}_{\A}=p\circ d_{\E} \circ s \in \Der_{\AWB}(\A).
\]
Moreover, since $p(d_{\E}(m))=0$ for any $m\in \M$, it follows that $d^{\E}_{\A}$ does not depend on the choice of the section $s$. Thus, we get a linear map
\[
\kappa \colon  \Der_{\AWB}(\E |\M) \lra \Der_{\AWB}(\M)\times \Der_{\AWB}(\A), \quad \kappa(d_{\E})=(d_{\E\mid\M}, d^{\E}_{\A}).
\]
It is easy to see that  $i \circ d_{\E\mid\M} = d_{\E} \circ i $ and $d^{\E}_{\A}  \circ  p =  p\circ d_{\E}$, that is,  $(d_{\E\mid\M}, d^{\E}_{\A})$ is an extensible pair in  $\Der_{\AWB}(\M)\times \Der_{\AWB}(\A)$. Then by Theorem~\ref{theorem extensible} we have that
$(d_{\E\mid\M}, d^{\E}_{\A})\in \D_{\AWB}(\M, \A)$ and $w(d_{\E\mid\M}, d^{\E}_{\A})=0$. Moreover, we have the following result

\begin{lemma}\label{lemma image kappa} With the above notations, the following assertions hold:
	\begin{itemize}
		\item[(i)] $\Im(\kappa)\subseteq  \D_{\AWB}(\M, \A)$;
		\item[(ii)] $\kappa \colon  \Der_{\AWB}(\E |\M) \to \D_{\AWB}(\M, \A) $ is a homomorphism of Lie algebras;
		\item[(iii)] $\Im(\kappa) = \Ker (\omega)$.
	\end{itemize}
\end{lemma}
\begin{proof}
	(i) There is nothing to prove because, as we have shown above, $\kappa(d_{\E})\in \D_{\AWB}(\M, \A)$ for all $d_{\E}\in \Der_{\AWB}(\E| \M)$.
	
	(ii) We need to show that $\kappa\big(d_{\E}d'_{\E}-d'_{\E}d_{\E}\big) = \kappa(d_{\E})\kappa(d'_{\E} )- \kappa(d'_{\E})\kappa(d_{\E}) $ for all $d_{\E}, d'_{\E}\in \Der_{\AWB}(\E | \M)$. For this, it suffices to repeat the proof given in the case of associative algebras in \cite[Lemma 5.4]{TX}.
	
	(iii) Since $w(\kappa(d_{\E}))= w(d_{\E\mid\M}, d^{\E}_{\A})=0$ for all $d_{\E}\in \Der_{\AWB}(\M, \A)$, we have $\Im(\kappa) \subseteq \Ker (\omega)$.
	So we must show that $\Ker (\omega)\subseteq \Im(\kappa)$. For that, take $(d_{\M}, d_{\A})\in \Ker (\omega)$. Then $\theta \big((d_{\M}, d_{\A}), \cl(f,g) \big)=0$,
	where $(f, g)$ is the $0$-cocycle in $K^*(\A,\M)$ induced by the given extension ${E}$ as in \eqref{eq_1cocycle(f,g)}.
	So, there exists a linear map  \big($-1$-cochain in $K^*(\A,\M)$\big) $h\colon \A\to \M$ such that $\theta \big((d_{\M}, d_{\A}), cl(f,g) \big)=\cl (\pa^{-1}(h))$,
	i.e. \eqref{eq_f and h} and \eqref{eq_g and h} hold. We define $d_{\E}\colon \E\to \E$ in the same way as in \eqref{def_d_E},  $d_{\E}(m+s(a))=d_{\M}(m)-h(a) +s(d_{\A}(a))$.
	One can repeat the respective part of the proof of Theorem~\ref{theorem extensible} to show that $d_{\E} \in \Der_{\AWB}(\E)$.
	Moreover, it is obvious that $d_{\E}(\M)\subseteq \M$ and hence, $d_{\E} \in \Der_{\AWB}(\E|\M)$.
	At the same time, ${d_{\E}}_{\mkern 1mu \vrule height 2ex\mkern2mu {\M}} = d_{\M}$ and $p\circ d_{\E}\circ s = d_{\A}$, which means that $\kappa(d_{\E})=(d_{\A}, d_{\E})$ and the proof is completed.
\end{proof}

Let us denote by $Z^{-1}(\A, \M)$ the vector space of all $-1$-cocycles in the cochain complex $K^*(\A, \M)$, i.e. it consists of all linear maps $h\colon \A\to \M$ satisfying
\begin{align}\label{0 cocycle h}
	a_0\cdot h(a_1)-h(a_0a_1)+h(a_0)\cdot a_1=0
\end{align}
for all $a_0, a_1\in \A$.

\begin{theorem}\label{theorem Well seq}
	Let $E \colon  0\lra \M \overset{i}{\lra} \E\overset{p}{\lra} \A\lra 0$ be an abelian extension of an \textup{AWB} $\A$ by its representation $\M$. Then there is an exact sequence of vector spaces
	\begin{equation}\label{Well seq}
		0\lra  Z^{-1}(\A, \M) \overset{\sigma}\lra \Der_{\AWB}(\E|\M) \overset{\kappa}{\lra} \D_{\AWB}(\M, \A)\overset{\omega}{\lra} H^{1}_{\AWB}(A,M) .
	\end{equation}
\end{theorem}
\begin{proof}
	Thanks to Lemma~\ref{lemma image kappa} (iii), it suffices to construct an injection of vector spaces $\sigma\colon  Z^{-1}(\A, \M) \to \Der_{\AWB}(\E|\M) $ and show exactness at the term  $\Der_{\AWB}(\E|\M)$.
	
	We define $\sigma$ as follows. For any $h\in Z^{-1}(\A, \M) $, let $\sigma (h) = d_{\E}$, where $d_{\E}\colon  \E\to \E$ is given by $d_{\E}(m+s(a))=h(a)$, for any
	$m\in \M$, $a\in \A$. Clearly ${d_{\E}}_{\mkern 1mu \vrule height 2ex\mkern2mu {\M}} = 0$ and $d^{\E}_{\A}(a) = (p\circ d_{\E} \circ s)(a)= p (h(a))=0$. Moreover, using \eqref{0 cocycle h}, it is easy to check that $d_{\E}$ is an \textup{AWB}-derivation. Hence $\sigma (h) = d_{\E}\in \Der_{\AWB}(\E|\M)$ and $\kappa(d_{\E}) =(0,0)$, i.e. $\Im(\sigma)\subseteq \Ker(\kappa) $. Obviously $\Ker(\sigma)=0$, so it is an injection.
	
	It remains to check that $\Ker(\kappa) \subseteq \Im(\sigma)$. For that, take $d_{\E}\in \Der_{\AWB}(\E|\M)$ such that
	$\kappa(d_{\E})=(0,0)$, i.e. ${d_{\E}}_{\mkern 1mu \vrule height 2ex\mkern2mu {\M}}=0$ and $p\circ d_{\E} \circ s =0$.
	 For any $a\in \A$ we get $d_{\E}(s(a))\in \Ker(p) = \M$. Thus there is a linear map $h\colon \A\to \M$ given by $h =  d_{\E} \circ s$.
	  One can easily check that $h$ does not depend on the choice of the linear section $s$, it satisfies condition
	\eqref{0 cocycle h}, i.e. $h\in Z^{-1}(\A, \M) $, and $\sigma(h)= d_{\E}$.
\end{proof}

\begin{corollary}\label{Cor Well seq}
	If the abelian extension $ 0\lra \M \overset{i}{\lra} \E\overset{p}{\lra} \A\lra 0$ is split, i.e. there is a homomorphism of \textup{AWB}'s $s\colon \A\to \E$ such that $p\circ s = \id_{\A}$,
	then there is an isomorphism of Lie algebras
	\[
	\Der_{\AWB}(\E|\M) \cong \D_{\AWB}(\M, \A)\times Z^{-1}(\A, \M),
	\]
	where $Z^{-1}(\A, \M)$ is considered as an abelian Lie algebra with trivial Lie bracket.
\end{corollary}
\begin{proof}
	As noted in Remark~\ref{remark w=0}, under the split condition $w=0$. Then the exact sequence~\eqref{Well seq} has the form
	\[
	0\lra  Z^{-1}(\A, \M) \overset{\sigma}\lra \Der_{\AWB}(\E|\M) \overset{\kappa}{\lra} \D_{\AWB}(\M, \A)\lra 0 .
	\]	
	Moreover, it is a short exact sequence of Lie algebras thanks to Lemma~\ref{lemma image kappa}  and by the fact that  $\sigma(h)\sigma(h')=0$ for all $h, h' \in Z^{-1}(\A, \M) $.
	
	For any $(d_{\M}, d_{\A}) \in  \D_{\AWB}(\M, \A)$ let us define $d_{\E}\colon {\E}\to \E$ by $d_{\E}(m+s(a))= d_{\A}(m)+ s(d_{\A}(a))$ for all $m\in \M$ and $a\in \A$.
	In the case  $w=0$, similarly to the proof of Theorem~\ref{theorem extensible}, it is easy to check that $d_{\E} \in \Der_{\AWB}(\E|\M) $. Thus we get a linear map
	$\D_{\AWB}(\M, \A) \to  \Der_{\AWB}(\E|\M)$, which actually is a homomorphism of Lie algebras and right inverse to $\kappa$. The details are direct computations, and we omit them.
\end{proof}

\section{Crossed modules of \textup{AWB}'s}

\subsection{Crossed modules}

\begin{definition}\label{def crossed module}
	A crossed module of  \textup{AWB}'s is a homomorphism of \textup{AWB}'s $\mu \colon  {\M} \to {\A}$ together with an action of $\A$ on $\M$ such that the following identities hold:
	\begin{enumerate}
		\item[(CM1)]
		$\begin{array} {lcllcl}
			\mu(m{ \cdot a})& = &\mu(m) a, & \mu({{a \cdot}} m) &= &a \mu(m),\\
			\mu\{m, a\}& = & [\mu(m), a],  & \mu\{a,m\}& = &[a, \mu(m)];
		\end{array}$
		\item[(CM2)]
		$\begin{array}{lcccl}
			{{\mu(m)  \cdot}} m' &=& m m' &=& m{ \cdot {\mu(m')}},\\
			
			\{\mu(m), m'\} &=& [m, m'] &=&  \{m, \mu(m')\}
		\end{array}$
	\end{enumerate}
	for all $m, m' \in {\M}$, $a \in{\A}$.
\end{definition}

\begin{definition}
	A morphism of crossed modules $\left({\M} \overset{\mu}\lra {\A} \right) \to \left({\M}' \overset{\mu'}\lra {\A}' \right)$
	is a pair $(\alpha, \beta)$, where $\alpha \colon  {\M} \to {\M}'$ and $\beta \colon  {\A} \to {\A}'$ are homomorphisms of \textup{AWB}'s satisfying:
	\begin{enumerate}
		\item[(a)] $\beta \circ \mu = \mu' \circ \alpha$.
		\item[(b)]
		$\begin{array}{lclclcl}
			\alpha({{a \cdot} m}) &=& {{\beta(a) \cdot}} \alpha(m), & & \alpha(m{ \cdot a}) &=&   \alpha(m){ \cdot {\beta(a)}},\\
			\alpha\{a,m\} &=& \{\beta(a), \alpha(m)\}, & & \alpha\{m, a\} &=&   \{\alpha(m), \beta(a)\}
		\end{array}$
	\end{enumerate}
	for all $a \in {\A}$, $m \in {\M}$.
\end{definition}

It is clear that crossed modules of \textup{AWB}'s constitute a category, denoted by {\sf XAWB}.

\

The following lemma is an easy consequence of Definition~\ref{def crossed module}.
\begin{lemma} \label{lemma cm}
	Let $\mu \colon  {\M} \to {\A}$ be a crossed module of algebras with bracket. Then the following statements are satisfied:
	\begin{enumerate}
		\item[(i)] $\Ker(\mu) \subseteq {\Z}(\M)$.
		\item[(ii)]  $\Im(\mu)$ is a two-sided ideal of ${\A}$.
		\item[(iii)]  $\Im(\mu)$ acts trivially on ${\Z}(\M)$, and so trivially on $\Ker(\mu)$. Hence $\Ker(\mu)$ inherits an action of ${\A}/\Im(\mu)$ making $\Ker(\mu)$ a representation of the algebra with bracket ${\A}/\Im(\mu)$.
	\end{enumerate}
\end{lemma}

\begin{example}
	\item[(i)] Let ${\A}$ be an \textup{AWB} and ${\B}$ be a two-sided ideal of ${\A}$, then the inclusion ${\B} \hookrightarrow {\A}$  is a crossed module, where the action of ${\A}$ on ${\B}$ is given by the operations in ${\A}$ (see Example~\ref{action}  {(ii)}).  Conversely, if $\mu \colon  {\B} \to {\A}$ is a crossed module of \textup{AWB}'s and $\mu$ is injective map, then ${\B}$ is isomorphic to a two-sided ideal of ${\A}$ by Lemma~\ref{lemma cm} { (ii)}.
	
	\item[(ii)] For any representation ${\M}$ of an \textup{AWB} ${\A}$, the trivial map $0 \colon {\M} \to {\A}$ is a crossed module with the action of ${\A}$ on the abelian \textup{AWB} ${\M}$ described in Example~\ref{action} {(i)}.
	
	Conversely, if $0 \colon  {\M} \to {\A}$ is a crossed module of \textup{AWB}'s, then ${\M}$ is necessarily an abelian \textup{AWB} and the action of ${\A}$ on ${\M}$
	is equivalent to  ${\M}$ being a representation of ${\A}$.
	
	\item[(iii)] Any homomorphism of \textup{AWB}'s $\mu \colon  {\M} \to  {\A}$, with ${\M}$ abelian and $\Im(\mu) \subseteq {\Z}({\A})$, provides a crossed module with ${\A}$ acting trivially on ${\M}$.
	
	\item[(iv)]  If $\mu \colon {\M}\to {\A}$ is a surjective homomorphism of \textup{AWB}'s and $\Ker(\mu) \subseteq {\Z}(\M)$,
	then $\mu$ is a crossed module with the induced action of ${\A}$ on ${\M}$ (see Example~\ref{action} {(v)}).
\end{example}

\begin{proposition}  Let $\mu \colon  {\M} \to {\A}$ be a crossed module of \textup{AWB}'s. Then the maps
	\begin{enumerate}
		\item[(i)] $(\mu, \id_{\A}) \colon  {\M} \rtimes {\A} \to {\A} \rtimes {\A}$,
		\item[(ii)] $(\id_{\M}, \mu) \colon  {\M} \rtimes {\M} \to {\M} \rtimes {\A}$,
		\item[(iii)]  $\varphi \colon  {\M} \rtimes {\A} \to {\M} \rtimes {\A}$ given by $\varphi(m, a) = (-m, \mu(m) + a)$,
	\end{enumerate}
	are homomorphisms of \textup{AWB}'s.
\end{proposition}
\begin{proof}
	(i) is a direct consequence of equalities in (CM1) of Definition~\ref{def crossed module}, (ii) follows from equalities in (CM2), whilst (iii) requires both (CM1), (CM2).
\end{proof}

\

At the end of this section, we show that, as is always the case, crossed modules are equivalent to internal categories in the category $\AWB$.

An internal category $({\B},{\A},s,t,\sigma,\gm)$ in $\AWB$ is a diagram of \textup{AWB}'s of the form
\[
\xymatrix@C=0.5cm{
	\B\times_{\A}\B\ar[r]^{\ \ \ \gm}& \B \ar@<0.6mm>[rr]^{s}\ar@<-0.6mm>[rr]_{t} &&
	\A\ar@/_1.3pc/[ll]_{\sigma}
}
\]
such that $s\sigma=t\sigma=\id_{\A}$ and the ``operation''
$\gm$ satisfies the usual axioms of a category. Here $\B\times_{\A}\B= \{(b_1, b_2)\in \B\times \B \mid t(b_1)=s(b_2)  \}$. A morphism of internal categories
$(\B,\A,s,t,\sigma,\gm)\to (\B',\A',s',t',\sigma',\gm')$ is a pair of homomorphisms of \textup{AWB}'s
$(\B\overset{\vp}\lra\B',\A\overset{\psi}\lra\A')$ such that  $\psi s=s' \vp$,
$\psi t=t'\vp$, $\vp\sigma=\sigma' \psi$ and $\vp\gm = \gm'
(\vp\times \vp)$. Let us denote by $\IAWB$ the respective category of internal categories in $\AWB$.

\begin{theorem}\label{Theo inter cat}
	The categories $\XAWB$ and $\IAWB$ are equivalent.
\end{theorem}

\begin{proof}
	
	To any object $({\B},{\A},s,t,\sigma,\gm)$ of the category $\IAWB$, we associate a crossed module
	$\mu\colon  \M\to \A$, where $\M=\Ker(s)$, $\mu=t_{\mkern 1mu \vrule height 2ex\mkern2mu {\M}}$ and the action of $\A$ on
	$\M$ is induced by the homomorphism $\sigma$. Since $t$ is a homomorphism of \textup{AWB}'s and $t(\sigma(a))=a$ for $a\in \A$, then all equalities in (CM1) of Definition~\ref{def crossed module} are trivially verified.
	Moreover, since $s\sigma = t \sigma=\id_{\A}$, by unities properties of the given
	internal category, we get
	\begin{align*}
		& \gm\Big(m+\sigma(a), \sigma(\mu(m)+a)+m'\Big) = \gm\Big(m+\sigma(a) + \sigma s(m'),
		\sigma t(m+ \sigma(a) )+m'\Big)\\
		& \qquad =\gm\big(m+\sigma(a), \sigma t(m+ \sigma(a))\big)+\gm\big(\sigma s (m'), m'\big)=m+\sigma(a)  +m'
	\end{align*}
	for all $a\in \A$ and $m,m'\in \M=\Ker(s)$. This equality and the fact that $\gamma$ is a homomorphism of \textup{AWB}'s, ensures all identities in (CM2).
	For instance, since
	\begin{align*}
		& \gm\Big[ \big(m_1 +\sigma(a_1), \sigma(\mu(m_1)+a_1)+m{_1}' \big), \big(m_2+\sigma(a_2), \sigma(\mu(m_2)+a_2)+m{_2}' \big) \Big] \\
		& \qquad = [m_1+\sigma(a_1)  +m{_1}', m_2+\sigma(a_2)  +m{_2}'],
	\end{align*}
	by considering $a_1=a_2=0$ and $m{_1}'=m_2=0$, we get
	\[
	\{\mu(m_1), m{_2}'\}= [\sigma\mu(m_1), m{_2}']= [m_1, m{_2}'].
	\]
	Hence $\mu\colon  \M\to \A$ indeed is a crossed module of \textup{AWB}'s.
	
	Moreover, if $(\vp, \psi)\colon (\B,\A,s,t,\sigma,\gm)\to (\B',\A',s',t',\sigma',\gm')$ is a morphism of internal categories, then it is easy to see that $(\vp_{\mkern 1mu \vrule height 2ex\mkern2mu {\M}}, \psi)\colon  \left({\M} \overset{\mu}\lra {\A} \right) \to \left({\M}' \overset{\mu'}\lra {\A}' \right) $
	is a morphism of the corresponding crossed modules. Thus this assignment defines a functor $\Psi\colon  \IAWB \to \XAWB$.
	
	Conversely, we now construct a functor  $\Psi'\colon  \XAWB \to \IAWB$, which will be a quasi-inverse of the functor $\Psi$.
	
	Let $\mu\colon \M\to\A$ be a crossed module of \textup{AWB}'s. Consider the semi-direct product $\M\rtimes \A$ together with maps
	$s,t\colon  \M\rtimes \A\to \A$ given by $s(m,a)=a$ and $t(m,a)=\mu(m)+a$. $s$ obviously is a
	homomorphism of \textup{AWB}'s. $t$ also is a homomorphism of \textup{AWB}'s since by (CM1) we get
	
	\begin{align*}
		t\big[ (m_1, a_1), (m_2, a_2) \big] &=  \mu[m_1, m_2] + \mu\{a_1, m_2\} + \mu\{m_1, a_2\} + [a_1, a_2]\\
		&	= [\mu (m_1), \mu( m_2)] + [a_1, \mu(m_2)] + [\mu(m_1), a_2] + [a_1, a_2] \\
		& = \big[\mu(m_1) + a_1, \mu(m_2) + a_2 \big] = \big[t(m_1, a_1), t(m_2, a_2) \big]
	\end{align*}
	And similarly $	t\big((m_1, a_1) (m_2, a_2) \big) = t(m_1, a_1) t(m_2, a_2)$. Moreover, applying the crossed module conditions
	(CM1) and (CM2), it is straightforward to see that the
	diagram
	\[	
	\xymatrix@C=0.5cm{
		(\M\rtimes\A)\times_{\A}
		(\M\rtimes\A)\ar[r]^{\qquad \ \ \ \gm}&
		\M\rtimes\A
		\ar@<0.6mm>[rr]^{s}\ar@<-0.6mm>[rr]_{t} &&
		\A\ar@/_1.3pc/[ll]_{\sigma} \;,
	}
	\]
	is an internal category in $\AWB$ whose ``operation'' $\gm$ is
	given by $\gm\big((m,a),(m',\mu(m)+a)\big)=(m+m',a)$ and $\sigma(a)=(0,a)$.
	Here observe that any element of
	$(\M\rtimes\A)\times_{\A}
	(\M\rtimes\A)$ is of the form $\big((m,a),(m',\mu(m)+ a)\big)$.
	
	Finally, the assignment $(\M\overset{\mu}\to\A)\to (\A\rtimes \A,\A, s,t,\sigma,\gm)$ is clearly functorial and
	provides the required quasi-inverse for the functor $\Psi$.
\end{proof}

\section{Crossed extensions and the second cohomology of \textup{AWB}'s}

Given a crossed module of \textup{AWB}'s $\mu\colon \MM\to \AAA$ we consider $\A=\Coker \mu$
and $\M=\Ker \mu$. By Lemma~\ref{lemma cm} (iii), $\M$ has a
well-defined structure of a representation of the \textup{AWB} $\A$.

Now let $\A$ be any \textup{AWB} and $\M$
a representation of $\A$. {\em A crossed extension} of
$\A$ by $\M$ is an exact sequence of \textup{AWB}'s
\[
\mathcal{E} \colon  0\lra\M\overset{\sigma}\lra{\MM}\overset{\mu}\lra{\AAA}\overset{\pi}\lra\A\lra 0\;,
\]
such that $\mu\colon  \MM\to \AAA$ is a crossed module
and the induced structure of $\A$-representation on
$\M$ coincides with the given one.

Let
$\mathcal{E'} \colon  0\lra\M\overset{\sigma'}\lra {\MM'}\overset{\mu'}\lra \AAA' \overset{\pi'}\lra\A\lra
0$ be another crossed extension of $\A$ by
$\M$. {\em A morphism of crossed extensions} from
$\mathcal{E}$ to $\mathcal{E}'$ is a morphism of crossed modules
of \textup{AWB}'s $(\al,
\beta)\colon (\MM\overset{\mu}\lra \AAA)\to
(\MM'\overset{\mu'}\lra \AAA')$ such that the
diagram
\[
	\xymatrix{
		0 \ar[r] & \M \ar@{}|{\parallel}[d] \ar[r]^{\sigma} & \ \MM \ar[d]_{\al}
		\ar[r]^{\mu} & \AAA \ar[d]_{\beta} \ar[r]^{\pi} & \A \ar@{}|{\parallel}[d] \ar[r] &
		0 \\
		0 \ar[r] & \M \ar[r]^{\sigma'} & \MM' \ar[r]^{\mu'} & \AAA' \ar[r]^{\pi'} & \A
		\ar[r] & 0   }
\]
is commutative.

Given two crossed extensions $\mathcal{E}$ and $\mathcal{E}'$ of
$\A$ by $\M$, they are said to be  {\em
	elementary equivalent} if there is a morphism from one to the
other. We consider the equivalence relation generated by the elementary equivalence in the set of all crossed extensions.
Let
$\XExt(\A,\M)$ denote the set of equivalence
classes of crossed extensions of the \textup{AWB}
$\A$ by the representation $\M$ of $\A$.
The set $\XExt(\A,\M)$ is not empty since it contains the class of the trivial crossed extension
\begin{equation}\label{trivial crossed extension}
	\mathcal{E}_0 \colon  0\lra\M\overset{=}\lra {\M}\overset{0}\lra \A \overset{=}\lra\A\lra 0.
\end{equation}

\

\begin{theorem}\label{T11}
	For any \textup{AWB} $\A$ and a representation $\M$ of $\A$ there is a canonical
	bijection
	\[
	\eta\colon {\XExt}(\A,\M)\overset{\approx}\lra
	H_{\AWB}^2(\A,\M).
	\]
	Moreover, $\eta$ maps the class of the trivial crossed extension $\mathcal{E}_0$ to the zero element in $H_{\AWB}^2(\A,\M)$.
\end{theorem}
\begin{proof} We define $\eta\colon \XExt(\A,\M)\to
	H_{\AWB}^2(\A,\M)$ as follows. Given
	\[
	\mathcal{E} \colon  0\lra\M\overset{\sigma}\lra \MM\overset{\mu}\lra \AAA \overset{\pi}\lra\A\lra
	0
	\]
	choose linear sections $s\colon \A\to \AAA$, $\pi
	s=\id_{\A}$ and $\rho\colon \Im \mu\to \MM$,
	$\mu\rho=\id_{\Im\mu}$.
	
	For any $a_0,a_1\in \A$ we have
	\[
	s(a_0)s(a_1)-s\left( a_0a_1\right)\in \Ker \pi=\Im \mu \quad \text{and} \quad [s(a_0),s(a_1)]-s[a_0,a_1]\in \Ker \pi=\Im \mu .
	\]
	Let $f\colon \A^{\otimes 2} \to \MM$ and $g\colon \A \to \Hom(\A, \MM)$ be given by
	\begin{align*}
		f(a_0,a_1) & =\rho\big( s(a_0)s(a_1)-s\left( a_0a_1\right) \big), \\
		g(a_0)(a_1) & = \rho\big([s(a_0),s(a_1)]-s[a_0,a_1]\big),
	\end{align*}
	and define
	\begin{align}\label{eq7}
		\mathsf{f}_{\mathcal{E}}(a_0, a_1, a_2)&= s(a_0)\cdot f(a_1,a_2)- f(a_0a_1,a_2)
		+f(a_0, a_1a_2) -f(a_0, a_1)\cdot s(a_2) ;
	\end{align}
	\begin{align}\label{eq7*}
		\mathsf{g}_{\mathcal{E}}(a_0,a_1)(a_2)&=s(a_0)\cdot \left( g(a_1)(a_2)\right) -g(a_0a_1)(a_2)+\left( g(a_0)(a_2)\right)\cdot s(a_1) \\ \notag
		& \quad -\{f(a_0,a_1),s(a_2)\} +f \left( [a_0,a_2],a_1 \right)  +f\left(a_0,[a_1,a_2]\right) .
	\end{align}
	
	Using the crossed module conditions (CM1) in Definition~\ref{def crossed module} and the fundamental identity \eqref{FE}, we directly verify
	that $\mu \mathsf{f}_{\mathcal{E}}(a_0, a_1, a_2)=0$ and $\mu \mathsf{g}_{\mathcal{E}}(a_0, a_1)(a_2)=0$. Thus
	we have defined linear maps
	$\mathsf{f}_{\mathcal{E}}\colon \A^{\otimes 3}\to \M$ and $\mathsf{g}_{\mathcal{E}}\colon \A^{\otimes 2}\to \Hom(\A,\M)$.
	By comparing \eqref{eq7} and \eqref{eq7*} respectively with \eqref{eq_0_cocycle_1} and \eqref{eq_0_cocycle_2},  we
	note that if only $f(a_0,a_1), \ g(a_0)(a_1)\in \M$, for all $a_0, a_1\in \A$, then  the definition of
	$(\mathsf{f}_{\mathcal{E}},\mathsf{g}_{\mathcal{E}} )$ would be read $(\mathsf{f}_{\mathcal{E}},\mathsf{g}_{\mathcal{E}} ) =\partial^0 (f,g)$  and
	hence would give $\partial^1(\mathsf{f}_{\mathcal{E}},\mathsf{g}_{\mathcal{E}} )=\partial^1\partial^0
	(f,g)=0$. Routine calculations show that $\partial^1(\mathsf{f}_{\mathcal{E}},\mathsf{g}_{\mathcal{E}} )$
	still is $0$ in our case too.
	Thus the pair $(\mathsf{f}_{\mathcal{E}},\mathsf{g}_{\mathcal{E}} )$ is a $1$-cocycle of the complex $K^*(\A,\M)$ and we define $\eta  \left( \mathcal{E}\right) $ to be the class of  $(\mathsf{f}_{\mathcal{E}},\mathsf{g}_{\mathcal{E}} )$ in
	$H_{\AWB}^2(\A,\M)=H^1\big(  K^*(\A,\M)\big) $.
	
	It requires standard but tedious calculations to show that $\eta$ is a well-defined map, i.e. the class of
	$(\mathsf{f}_{\mathcal{E}},\mathsf{g}_{\mathcal{E}} )$ does not depend on the sections $s$, $\rho$ and on a representative of the class of $\mathcal{E}$. We will leave this to the reader as an exercise.
	
	By construction, it is readily seen that $\eta(cl (\mathcal{E}_0))=0$, as required.
	
	In the converse direction we define a map $\eta'\colon H_{\AWB}^2(\A,\M)\to
	{\XExt}(\A,\M)$ as follows. Let $F(\A)$ be the free \textup{AWB} on the underlying vector space of $\A$ (for construction see~\cite[Proposition 3.1]{CP}), and
	$\nu\colon F(\A)\to \A$  the canonical projection. Then $\M$ can be considered as a representation of $F(\A)$ via  the homomorphism $\nu$. Moreover, we have the map of  cochain complexes
	$\nu^*\colon K^*(\A,\M)\to K^*(F(\A),\M)$ induced by $\nu$. In low dimensions, it gives the following diagram with commutative squares
	\begin{equation}
		\xymatrix{
			K^0(\A,\M) \ar[d]_{\nu^0} \ar[r]^{\partial^0} & K^1(\A,\M)
			\ar[d]_{\nu^1} \ar[r]^{\partial^1} & K^2(\A,\M) \ar[d]_{\nu^2}
			\ar[r]^{\qquad \partial^2} & \cdots \\
			K^0(F(\A),\M)  \ar[r]^{\delta^0} & K^1(F(\A),\M)
			\ar[r]^{\delta^1} & K^2(F(\A),\M)
			\ar[r]^{\qquad \ \ \delta^2} & \cdots }
	\end{equation}
	For any $1$-cocycle $(\mathsf{f},\mathsf{g})\in K^1(\A,\M)$ of the complex $K^*(\A,\M)$, since $\nu^1 (\mathsf{f},\mathsf{g})$ is a $1$-cocycle of
	$K^*(F(\A),\M)$ and the second cohomology of any free \textup{AWB} is trivial (see \cite[Lemma 4.7]{CP}), then we deduce that there exists a $0$-cochain
	$(f,g)\in K^0(F(\A),\M)$
	such that $\nu^1 (\mathsf{f},\mathsf{g}) = \delta^0 (f,g)$. Thus, for all $\ol a_1, \ol a_2, \ol a_3 \in F(\A)$ we have  $\delta^0 (f,g)(\ol a_1, \ol a_2, \ol a_3 ) = \big(\mathsf{f}(\ol a_1, \ol a_2, \ol a_3), \mathsf{g}(\ol a_1, \ol a_2)(\ol a_3)\big) $, which is equal to $(0, 0)$ as soon as one of $\ol a_i\in \Ker\nu$, $i=1,2,3$. Furthermore, since any element of $\Ker\nu$ acts trivially on the elements from $\M$, thanks to the equations \eqref{eq_0_cocycle_1} and \eqref{eq_0_cocycle_2},  we get immediately
	\begin{align}
		f(v_1v_2, v_3) & = f(v_1, v_2v_3), \label{eq_VM1}\\
		g(v_1 v_2)(v_3) & = f\big([v_1,v_3],v_2\big) +	f\big(v_1, [v_2,v_3]\big) \label{eq_VM2}
	\end{align}
	for all $v_1,v_2, v_3\in \Ker \nu$.
	
	Now we claim that the following sequence is a
	crossed extension of $\A$ by $\M$
	\begin{align}\label{eq13}
		&\mathcal{E}_{(\mathsf{f},\mathsf{g})}  \colon  0\lra\M\overset{\sigma}\lra\M\oplus\V\xrightarrow{\mu_{(\mathsf{f},\mathsf{g})}}
		F(\A)\overset{\nu}\lra\A\lra 0\;,
	\end{align}
	where $\V=\Ker \nu$, $\sigma(m)=(m,0)$; the \textup{AWB} structure on
	$\M\oplus\V$ is given by
	\begin{align}
		(m_1,v_1)(m_2,v_2) & = (f(v_1,v_2),v_1v_2), \label{eq_VM Asso} \\
		[(m_1,v_1),(m_2,v_2)] & =\big( g(v_1)(v_2),[v_1,v_2]\big) \;;\label{eq_VM AWB}
	\end{align}
	the action of $F(\A)$ on $\M\oplus\V$
	is given by
	\begin{align}\label{eq12}
		\bar{a} \cdot (m,v) & =\big(\nu (\bar{a}) \cdot m +f(\bar{a},v), \bar{a} v \big)\;,\\
		(m,v)\cdot\bar{a} & =\big(m\cdot \nu (\bar{a})+f(\bar{a},v), v\bar{a}\big)\;,\nonumber\\
		\{\bar{a},(m,v)\} & =\big(\{\nu (\bar{a}), m\}+g(\bar{a})(v),[\bar{a},v]\big)\;,\nonumber\\
		\{(m,v),\bar{a}\} & =\big(\{m,\nu (\bar{a})\}+g(\bar{a})(v),[v,\bar{a}]\big)\;,\nonumber
	\end{align}
	and the map $\mu_{(\mathsf{f},\mathsf{g})}$ is defined by $\mu_{(\mathsf{f},\mathsf{g})}(m,v)=v$, for every
	$m,m_i\in\M$, $v,v_i\in\V$ ($i=1,2$) and
	$\bar{a}\in F(\A)$.
	
	It is routine to show that \eqref{eq_VM Asso} and \eqref{eq_VM AWB} indeed define an \textup{AWB} structure on $\M\oplus\V$. For instance, the equality \eqref{eq_VM1} amounts to the associativity axiom in  $\M\oplus\V$, whilst \eqref{eq_VM1} shows that the fundamental equality \eqref{FE} holds.
	At the same time, checking that the equations in \eqref{eq12} really define an
	action of $F(\A)$ on $\M\oplus\V$ requires direct calculations and the fact that
	$\delta^0 (f,g)(\ol a_1, \ol a_2, \ol a_3 ) =(0,0)$ whenever $\ol a_i\in \V$ for some $i=1,2,3$.
	Further,  $\mu_{(\mathsf{f},\mathsf{g})}$ clearly satisfies the conditions (CM1), (CM2) and so it is a crossed module of \textup{AWB}'s.
	
	Thus $\mathcal{E}_{(\mathsf{f},\mathsf{g})}$ is a crossed extension of $\A$ by
	$\M$ and we define $\eta'(\mathsf{f},\mathsf{g})$ to be the class of
	$\mathcal{E}_{(\mathsf{f},\mathsf{g})}$ in $\XExt(\A,\M)$.
	
	Finally, omitting straightforward but tedious details, we  note that $\eta'$ is a well-defined map from
	$H_{\AWB}^2(\A,\M)$ to ${\XExt}(\A,\M)$, i.e. the class of
	$\mathcal{E}_{(\mathsf{f},\mathsf{g})}$ does not depend on a
	representative of the class of $(\mathsf{f},\mathsf{g})$ in $H_{\AWB}^2(\A,\M)$ and on the chosen pair $(f,g)\in K^0(F(\A),\M) $, and it is a two-sided inverse to $\eta$.
\end{proof}

The bijection in Theorem~\ref{T11} allows us to endow the set  ${\XExt}(\A,\M)$ with a vector space structure induced from the one of $H_{\AWB}^2(\A,\M)$.
The respective addition in ${\XExt}(\A,\M)$ can be defined by the  Baer sum of crossed extensions of $\A$ by $\M$ presented
in the next section.

\section{Eight-term exact sequence in the cohomology of \textup{AWB}'s}

The detailed proofs of results presented in this section require direct but careful checking. They are almost complete analogs of those given in \cite{Ra,Wa} for groups and Lie algebras, in \cite{Ca2} for Leibniz algebras, and will be largely omitted.

\subsection{Baer sum of crossed extensions}

Suppose we are given a crossed  extension $\mathcal{E} \colon  0\to\M\overset{\sigma}\lra{\MM}\overset{\mu}\lra{\AAA}\overset{\pi}\lra\A\to 0 $ of $\A$ by $\M$. Let $\psi \colon  \A'\to \A$ be a homomorphism of \textup{AWB}'s  and $\phi\colon \M\to \M'$ a homomorphism of $\A$-representations (i.e. a linear map preserving the actions of $\A$). One constructs the \emph{pull back crossed extension} $\mathcal{E}_{\psi}$ of $\A'$ by $\M$ induced by $\psi$, and  the \emph{push out crossed extension}  ${}^\phi\mathcal{E}$  of $\A$ by $\M'$ induced by $\phi$, as follows:
\[
\mathcal{E}_{\psi} \colon  0\to\M\overset{\sigma}\lra{\MM}\overset{\mu{_\psi}}\lra{\AAA_{\psi}}\overset{\pi{_\psi}}\lra\A'\to 0 ,
\]
where
\[
\AAA_{\psi} = \AAA \times_{\A} \A' = \{ (a, a')\in \AAA \times \A' \mid \pi(a)=\psi(a') \}, \quad   \ \pi_{\psi}(a, a')=a', \ \mu_{\psi}(m)= (\mu(m), 0),
\]
and the action of $\AAA_{\psi}$ on $\MM$ is given via the homomorphism $\AAA_{\psi} \to \AAA$, $(a, a')\mapsto a$;
\[
{}^\phi\mathcal{E} \colon  0\to\M'\overset{\sigma_{\phi}}\lra{\MM_{\phi}}\overset{\mu_{\phi}}\lra{\AAA}\overset{\pi}\lra\A\to 0,
\]
where 	
\begin{align*}
	\MM_{\phi} = \M'\times \MM / S, \ \text{with} \ & S=\{\big(\phi(m), -\sigma(m)\big) \mid m\in \M \}, \\
	&\sigma_{\phi}(m') = \cl(m', 0), \ \mu_{\phi}(\cl(m',m_0))= \mu(m_0), \ m'\in \M', \ m_0\in \MM,
\end{align*}
and the action of $\AAA$ on $\MM_{\psi}$ is induced by the action of $\AAA$ on $\MM$.

\begin{proposition}\label{XExt functor} \
	\begin{itemize}
		\item[(i)]	A homomorphism of \textup{AWB}'s  $\psi\colon  \A'\to \A$ induces a map $\psi_{*}\colon {\XExt}(\A,\M) \to {\XExt}(\A',\M) $, $\psi_*(\cl(\mathcal{E}))=\cl(\mathcal{E}_{\psi})$;
		\item[(ii)] A homomorphism of $\A$-representations $\phi\colon \M\to \M'$  induces a map
		$\phi^*\colon {\XExt}(\A,\M) \to {\XExt}(\A,\M') $, $\phi^*(\cl(\mathcal{E}))=\cl({}^{\phi}\mathcal{E})$.
	\end{itemize}	
\end{proposition}

\

Now suppose we are given two crossed extensions of an \textup{AWB} $\A$ by its representation $\M$
\begin{align*}
	\mathcal{E} \colon  0\to\M\overset{\sigma}\lra{\MM}\overset{\mu}\lra{\AAA}\overset{\pi}\lra\A\to 0 \ \text{and} \
	\mathcal{E'} \colon  0\to\M\overset{\sigma'}\lra{\MM'}\overset{\mu'}\lra{\AAA'}\overset{\pi'}\lra\A\to 0.
\end{align*}
Then
\[
\mathcal{E\times E'} \colon  0\to\M\times \M' \xrightarrow{\sigma\times \sigma'}{\MM\times \MM'}\xrightarrow{\mu\times \mu'} {\AAA\times \AAA'}\xrightarrow{\pi\times \pi'} \A\times \A'\to 0\
\]
is also a crossed extension with component-wise action of $\AAA\times \AAA'$ on $\MM\times \MM'$ making $\mu\times \mu'$ to be a crossed module of \textup{AWB}'s.

Let $\Delta\colon  \A\to \A\times \A$,  $\Delta(a)=(a,a)$ and  $\nabla\colon \M\times \M \to \M$, $\nabla(m_1, m_2)= m_1 + m_2$,
be the diagonal and codiagonal maps, respectively. Then the \emph{Baer sum} of $\mathcal{E}$ and $\mathcal{E}'$ is defined to be
\[
\mathcal{E}+\mathcal{E}' = {}^{\nabla}\big((\mathcal{E}\times\mathcal{E}')_{\Delta}\big),
\]
which is also crossed extension of $\A$ by $\M$. The set $\XExt(\A,\M)$ is a vector space with respect to the addition
\[
\cl(\mathcal{E}) + \cl(\mathcal{E}') = \cl (\mathcal{E} + \mathcal{E}'); 
\]
and scalar multiplication given for any $\lambda \in \K$ by
\[
\lambda \cl({\mathcal{E}}) = \cl(\lambda \mathcal{E}), \ \text{where} \ \lambda\mathcal{E} \colon  0\to\M\xrightarrow{\lambda \cdot\sigma}{\MM}\xrightarrow{\;\mu \;}{\AAA}\xrightarrow{\; \pi \;}\A\to 0.
\]
The neutral element in $\XExt(\A,\M)$ is the class of the trivial crossed extension $\mathcal{E}_0$ in \eqref{trivial crossed extension}.

\begin{remark}\label{Remark isomorphism} \
	\begin{itemize}
		\item[(i)] The bijection $\eta\colon {\XExt}(\A,\M)\overset{\approx}\lra H_{\AWB}^2(\A,\M)$ in Theorem~\ref{T11} is an isomorphism of vector spaces.
		\item[(ii)] The maps $\psi_{*}\colon {\XExt}(\A,\M) \to {\XExt}(\A',\M) $ and $\phi^*\colon {\XExt}(\A,\M) \to {\XExt}(\A,\M')$ in Proposition~\ref{XExt functor} are linear maps.
	\end{itemize}	
\end{remark}

\subsection{$\A$-crossed extensions}

Suppose $\A$ is an \textup{AWB}, $\B$  a two-sided ideal of $\A$ and $\M$  a representation of $\A\!/\!\B$.
Consider $\M$ as a representation of $\A$ with trivial action of $\B$ on $\M$.
An 	\emph{$\A$-crossed extension} of $\B$ by $\M$ is a short exact sequence of \textup{AWB}'s
\[
{\e} \colon 	0\lra \M \overset{i}{\lra} \C \overset{\nu}{\lra} \B \lra 0
\]
with an action of $\A$ on $\B$ by operations in $\A$ (see Example~\ref{action} (ii))
and an action of $\A$ on $\C$ such that $\nu$ is a crossed module of \textup{AWB}'s.
Such $\A$-crossed extension always exists, since the trivial extension
\[
{\sf e}_0 \colon 	0\lra \M {\lra} \M \times \B {\lra} \B \lra 0
\]
is an example, with component-wise action of $\A$ on $\M \times \B$.

Two  $\A$-crossed extensions $\e\colon  0\to \M \overset{i}{\to} \C \overset{\nu}{\to} \B \to 0$ and
$\e'\colon  0\to \M \overset{i'}{\to} \C' \overset{\nu'}{\to} \B \to 0 $ are called \emph{congruent}
if there exists a homomorphism $\varphi\colon  \C \to \C'$ such that
$(\varphi, \id_{\B})$ is a morphism of crossed modules and $\varphi_{\mkern 1mu \vrule height 2ex\mkern2mu {\M}}= \id_{\M}$.
This induces an equivalence relation on the set of all $\A$-crossed extensions of $\B$ by $\M$. Let $\XExt_{\A}(\B, \M)$ denote the set of equivalence classes.

Let ${\e}\colon 	0\to \M \overset{i}{\lra} \C \overset{\nu}{\lra} \B \to 0$ and ${\e}' \colon 	0\to \M \overset{i'}{\lra} \C' \overset{\nu'}{\lra} \B \to 0$
be $\A$-crossed extensions of $\B$ by $\M$. Then
\[
{\sf e} \times {\sf e}'\colon 	0\to \M\times \M \xrightarrow{i\times i'} \C \times \C' \xrightarrow{\nu\times \nu'} \B\times \B \to 0
\]
is an $\A\times \A$-crossed extension of $\B\times \B$ by $\M\times \M$. The diagonal map $\Delta\colon \B \to \B\times \B$, $\Delta(b)=(b,b)$ is a homomorphism of \textup{AWB}'s and induces the pull back $\A$-crossed extension
\[
\big({\e} \times {\e}'\big)_{\Delta}\colon 	0\to \M\times \M  {\lra}  \C_{\Delta} {\lra} \B \to 0,
\]
where $\C_{\Delta}= \{(c, c', b)\in \C\times \C\times \B \mid (\nu(c), \nu'(c'))=(b,b)\} $, with the component-wise action of $\A$ on $\C_{\Delta}$.
On the other hand, the codiagonal map $\nabla\colon  \M\times \M \to \M $, $\nabla(m_1,m_2) = m_1 +m_2 $
is a homomorphism of $\A$-representations and induces the push out $\A$-crossed extension
\[
{}^{\nabla}\big({\sf e} \times {\sf e}'\big)\colon 	0\to \M  {\lra} {}^{\nabla}\!\C {\lra} \B \times \B \to 0,
\]
where ${}^{\nabla}\!\C = \big(\M\times \C\times \C'\big) / S$,  with $S= \{ (m_1+m_2, - i(m_1), -i(m_2)) \mid m_1, m_2 \in \M \}$.

Then the \emph{Baer sum} of two $\A$-crossed extensions $\e$ and $\e'$ is defined to be
\[
\e+\e' = {}^{\nabla}\big((\e\times\e')_{\Delta}\big),
\]
and it is also an $\A$-crossed extension of $\B$ by $\M$.
Then the set $\XExt_{\A}(\B, \M)$ is a vector space with respect to the addition given by
\[
\cl(\e) + \cl(\e')=\cl(\e+\e'),
\]
and scalar multiplication  defined for any $\lambda\in \K$ by
\[
\lambda \cl(\e)= \cl{(\lambda\e)}, \ \text{where} \ \lambda\e\colon  0\to \M \xrightarrow{\lambda \cdot i} \C \xrightarrow{\; \nu  \;}\B \to 0 .
\]

Note that the neutral element with respect to the addition in  $\XExt_{\A}(\B, \M)$ is the equivalence class $\cl(\e_0)$ of the trivial extension
$\e_0$. Furthermore, if  $\cl(\e)=\cl(\e_0) =0$,
then the sequence ${\e}\colon 	0\to \M \overset{i}{\lra} \C \overset{\nu}{\lra} \B \to 0$ is split by an $\A$-homomorphism,
i.e. there exists a homomorphism of \textup{AWB}'s $s\colon \B\to \C$ which preserves the action of $\A$ and $\nu \circ s =\id_{\B}$.

\

Let $\B$ be a two-sided ideal of an \textup{AWB} $\A$ and $\M$  a representation of $\A\!/\!\B$. Consider $\M$ as a representation of $\A$ with trivial action of $\B$
on $\M$. Then any abelian extension of $\A$ by $\M$
\[
E\colon  0\to \M \overset{i}{\lra} \E \overset{p}{\lra} \A\to 0
\]
defines an $\A$-crossed extension
\[
\e_{E}\colon  0\to \M {\lra} \E\times_{\A} \B \overset{\pi_2}{\lra} \B\to 0,
\]
where $\E\times_{\A} \B = \{(e,b)\in \E\times \B  \mid p(e)=b  \} \cong p^{-1}(B)$, $\pi_2(e,b)=b$ and the action of $\A$ on $\E\times_{\A} \B$ is given by
\begin{align*}
	&a\cdot (e,b) = (e'e, ab),   & \{	a, (e,b) \}= ([e',e], [a,b]),\\
	& (e,b) \cdot a = (ee', ba), & \{	(e,b), a \}= ([e,e'], [b,a]),
\end{align*}
for all $(e,b)\in \E\times_{\A}\B$, $a\in \A$ and  $e'\in E$ is any element such that $p(e')= a$.

The correspondence $E\mapsto \e_E$ defines a linear map $r\colon  \Ext_{\AWB}(\A, \M) \to \XExt_{\A}(\B, \M)$, $r(\cl(E))= \cl(\e_E)$.

\begin{lemma}\label{lemma exact seq1}
	Given a two-sided ideal $\B$ of an \textup{AWB} $\A$ and a representation $\M$ of $\A\!/\!\B$, there is an exact sequence of vector spaces
	\begin{equation*}
		\Ext_{\AWB}(\A\!/\!\B, \M) \overset{\tau^{\star}}{\lra} \Ext_{\AWB}(\A, \M)\overset{r}{\lra} \XExt_{\A}(\B, \M).
	\end{equation*}
\end{lemma}
\begin{proof}
	The map $\tau^{\star}$ is induced by the canonical map $\tau\colon  \A\to \A\!/\!\B$. It is easy to see that the composition $r\circ \tau^{\star}$ is trivial. To show that $\Ker(r)\subseteq \Im(\tau^{\star})$, suppose we are given an extension
	$E\colon  0\to \M {\lra} \E \overset{p}{\lra} \A\to 0 $ which represents an element in $\Ext_{\AWB}(\A, \M)$ such that
	$r(\cl(E))= \cl(\e_E) =0$, i.e. $\e_{E}\colon  0\to \M {\lra} \E\times_{\A} \B \overset{\pi_2}{\lra} \B\to 0$ is split by an $\A$-homomorphism $s\colon B\to \E\times_{\A} \B$.   Then $\Im(\pi_1\circ s)$ is a two-sided ideal of the \textup{AWB} $\E$, where $\pi_1\colon  \E\times_{\A} \B\to \E $, $\pi_1(e,b)= e$. Let us denote $\E_s= \E/\Im(\pi_1\circ s)$. Then $p$ induces an epimorphism $p_s\colon  \E_s \to \A\!/\!\B $ whose kernel is $\M$. So we constructed an extension
	\[
	E_s \colon  0\to \M {\lra} \E_s \overset{p_s}{\lra} \A\!/\!\B \to 0
	\]
	of the \textup{AWB} $\A\!/\!\B$ by its representation $\M$ and it is readily seen that $\tau^{\star}(\cl(E_s)) = \cl(E)$.
\end{proof}

Let us remark that the exact sequence in Lemma~\ref{lemma exact seq1} can be rewritten as the following exact sequence of vector spaces
\begin{equation}\label{exact seq1}
	H^{1}_{\AWB}(\A\!/\!\B, \M) \overset{\tau^{\star}}{\lra} H^{1}_{\AWB}(\A, \M)\overset{r}{\lra} \XExt_{\A}(\B, \M).
\end{equation}

\subsection{Eight-term exact sequence}
As in the previous subsection, we suppose $\A$ is an \textup{AWB}, $\B$  a two-sided ideal of $\A$ and $\M$  a representation of $\A\!/\!\B$.
Then any $\A$-crossed extension ${\e} \colon 	0\lra \M \overset{i}{\lra} \C \overset{\nu}{\lra} \B \lra 0$
of $\B$ by $\M$ defines a crossed extension of the \textup{AWB} $\A\!/\!\B$ by its representation $\M$
\[
{\EE}_{\e} \colon 	0\lra \M \overset{i}{\lra} \C \overset{\nu}{\lra} \A \overset{\tau}\lra  \A\!/\!\B \lra 0.
\]
This correspondence defines a linear map
\[
\gamma \colon  \XExt_{\A}(\B, \M) \lra \XExt(\A\!/\!\B, \M), \quad \gamma(\cl(\e))=\cl({\EE}_{\e})
\]

\begin{lemma}\label{proposition exact seq2 }
	Given a two-sided ideal $\B$ of an \textup{AWB} $\A$ and a representation $\M$ of $\A\!/\!\B$, there is an exact sequence of vector spaces
	\begin{equation*}
		\Ext_{\AWB}(\A, \M)\overset{r}{\lra} \XExt_{\A}(\B, \M)  \overset{\gamma}{\lra} \XExt(\A\!/\!\B, \M)\overset{\tau^*}{\lra} \XExt(\A, \M),
	\end{equation*}
	where $\tau^*$ is induced by the canonical map $\tau\colon \A\to \A\!/\!\B$.
\end{lemma}
\begin{proof} We leave the proof of the exactness in $\XExt_{\A}(\B, \M)$ as an exercise, only noting that it is similar to the case of Leibniz algebras
	\cite[Theorem 2]{Ca2}. So, we will prove that $\Im(\gamma)= \Ker(\tau^*)$.
	
	First, show that  $\gamma \circ \tau^* =0$. Take  ${\e} \colon 	0\lra \M \overset{i}{\lra} \C \overset{\nu}{\lra} \B \lra 0$
	to be an $\A$-crossed extension which represents an element in  $\XExt_{\A}(\B, \M) $. Then $(\gamma \circ \tau^*) (\cl(\e)) = \cl\big((\EE_{\e})_{\tau}\big) $, where
	\[
	(\EE_{\e})_{\tau} \colon 	0\lra \M \overset{i}{\lra} \C \overset{\nu_{\tau}}{\lra} \A\times_{\A\!/\!\B} \A \overset{\pi_{\tau}}\lra  \A \lra 0,
	\]
	with $\nu_{\tau}(c)=(\nu(c), 0)$ and $\pi_{\tau}(a,a')=a'$ for all $c\in \C$ and $(a,a')\in \A\times_{\A\!/\!\B} \A$.
	Since the diagram
	\[
		\xymatrix{
			\EE_{0}\colon 	0 \ar[r] & \M \ar@{}|{\parallel}[d] \ar[r]^{=} & \ \M \ar[d]_{i}
			\ar[r]^{0} & \quad \A \quad \ar[d]_{\beta} \ar[r]^{=} & \A \ar@{}|{\parallel}[d] \ar[r] &
			0 \\
			(\EE_{\e})_{\tau}\colon 	0 \ar[r] & \M \ar[r]^{i} & \C \ar[r]^{\nu_{\tau} \quad } &  \A\times_{\A\!/\!\B} \A \ar[r]^{\quad \pi_{\tau}} & \A
			\ar[r] & 0   }
	\]
	is commutative, where $\beta(a)=(a,a)$, and $(i, \beta)\colon (\M\overset{0}{\lra}\A)\to (\C \overset{\nu_{\tau}}{\lra}\A\times_{\A\!/\!\B} \A )$
	is a morphism of crossed modules, we deduce that $\cl\big((\EE_{\e}\big)_{\tau})= \cl(\EE_{0})=0$.
	
	Now suppose $\EE \colon  0\to\M\overset{\sigma}\lra{\MM}\overset{\mu}\lra{\AAA}\overset{\pi}\lra\A\!/\!\B\to 0$
	is a crossed extension of $\A\!/\!\B$ by $\M$ which represents an element in $ \XExt(\A\!/\!\B, \M)$  such that $\tau^*(\cl(\EE))=0$. 
	This means that the crossed extension
	\[
	\EE_{\tau} \colon  0\to\M\overset{\sigma}\lra{\MM}\overset{\mu_{\tau}}\lra{\AAA\times_{\A\!/\!\B} \A}\overset{\pi_{\tau}}\lra\A\to 0
	\]
	is split, i.e. there is a homomorphism of \textup{AWB}'s $s\colon  \A\to \AAA\times_{\A\!/\!\B} \A$ such that $\pi_{\tau}\circ s = \id_{\A}$.
	Let $\pi_1\colon  \AAA\times_{\A\!/\!\B} \A \to \AAA$ be the canonical projection. 
	Then we denote by $h$ the restriction  $\pi_1\circ s_{\mkern 1mu \vrule height 2ex\mkern2mu {\B}}$.
	Clearly, $\Im(h)\subseteq \Ker(\pi)=\Im(\mu)$. So, $h\colon \B\to \Im(\mu)$  is a homomorphism of \textup{AWB}'s. It is easy to see that
	\[
	\e_{\EE}\colon  0\to \M {\lra} \MM \times_{\Im(\mu)} \B \overset{\pi_2}{\lra} \B\to 0
	\]
	is an $\A$-crossed extension, with the action of $\A$ on $\MM \times_{\Im(\mu)} \B $  given by
	\begin{align*}
		&a\cdot (m, b)= (\pi_1s(a)\cdot m, ab), \quad & \{a, (m, b)\}= \big(\{\pi_1s(a), m\}, [a, b]\big), \\
		& (m, b)\cdot a = (m \cdot \pi_1s(a), ba), \quad & \{(m, b), a\} = \big(\{m, \pi_1s(a)\}, [ b,a]\big),
	\end{align*}
	for $a\in \A$, $(m, b) \in \MM \times_{\Im(\mu)} \B$, and we have $\gamma(\cl(\e_{\EE}))= \cl(\EE)$. Hence $\Ker(\tau^*)\subseteq \Im (\gamma)$.
\end{proof}

Taking in mind the isomorphisms in Remark~\ref{Remark isomorphism} (i) and \eqref{iso ext h1}, we get immediately the following exact sequence of vector spaces
\begin{equation}\label{exact seq2 }
	H^{1}_{\AWB}(\A, \M)\overset{r}{\lra}	\XExt_{\A}(\B, \M)  \overset{\gamma}{\lra} H^{2}_{\AWB}(\A\!/\!\B, \M) \overset{\tau^*}{\lra} H^{2}_{\AWB}(\A, \M).
\end{equation}

To state the last result we need to recall from \cite{Ca} that $\B_{\ab}= \B/[[\B,\B]]$ has a natural structure of an $\A\!/\!\B$-representation  and  $\Hom_{\A\!/\!\B}(\B_{\ab}, \M)$ denotes the vector space of all linear maps from $\B_{\ab}$ to $\M$ preserving the action of $\A\!/\!\B$.

\begin{theorem}\label{Theo 8 term }
	Let $0\to \B\to \A\overset{\tau}\to \A\!/\!\B \to 0$ be a short exact sequence of \textup{AWB}'s and $\M$ be a representation of $\A\!/\!\B$.  Then there is an exact sequence of vector spaces
	\begin{align*}
		0\lra H^0_{\AWB}(\A\!/\!\B, \M) \lra H^0_{\AWB}(\A, \M) \lra \Hom_{\A\!/\!\B}(\B_{\ab}, \M) \lra H^{1}_{\AWB}(\A\!/\!\B, \M) \\
		\overset{\tau^{\star}}\lra
		H^{1}_{\AWB}(\A, \M)\overset{r}{\lra}	\XExt_{\A}(\B, \M)  \overset{\gamma}{\lra} H^{2}_{\AWB}(\A\!/\!\B, \M) \overset{\tau^*}{\lra} H^{2}_{\AWB}(\A, \M).
	\end{align*}
\end{theorem}
\begin{proof}
	This is a combination of \cite[Exact Sequence (3) and Theorem 2.9]{Ca}, isomorphism~\eqref{iso h0} and exact sequences \eqref{exact seq1} and \eqref{exact seq2 }.
\end{proof}

\subsection*{Acknowledgements}

The authors were supported by  Agencia Estatal de Investigaci\'on (Spain), grant PID2020-115155GB-I00.
Emzar Khmaladze was financially supported by EU fellowships for Georgian researchers, 2023 (57655523) and by  Shota Rustaveli National Science Foundation of Georgia, grant FR-22-199. He is grateful to the University of Santiago de Compostela for its hospitality. Third author was also supported by Xunta de Galicia
through the Competitive Reference Groups (GRC), ED431C 2019/10.

\end{document}